\documentclass[11pt]{article}
\usepackage{latexsym,amsfonts,amssymb,amsmath,amsthm}

\usepackage{color}

\begin{document}

\newtheorem{tm}{Theorem}[section]
\newtheorem{prop}[tm]{Proposition}
\newtheorem{defin}[tm]{Definition}
\newtheorem{coro}[tm]{Corollary}

\newtheorem{lem}[tm]{Lemma}
\newtheorem{assumption}[tm]{Assumption}

\newtheorem{rk}[tm]{Remark}

\newtheorem{nota}[tm]{Notation}
\numberwithin{equation}{section}

\newcommand{\stk}[2]{\stackrel{#1}{#2}}
\newcommand{\dwn}[1]{{\scriptstyle #1}\downarrow}
\newcommand{\upa}[1]{{\scriptstyle #1}\uparrow}
\newcommand{\nea}[1]{{\scriptstyle #1}\nearrow}
\newcommand{\sea}[1]{\searrow {\scriptstyle #1}}
\newcommand{\csti}[3]{(#1+1) (#2)^{1/ (#1+1)} (#1)^{- #1
 / (#1+1)} (#3)^{ #1 / (#1 +1)}}
\newcommand{\RR}[1]{\mathbb{#1}}

\newcommand{ \bl}{\color{blue}}
\newcommand {\rd}{\color{red}}
\newcommand{ \bk}{\color{black}}
\newcommand{ \gr}{\color{OliveGreen}}
\newcommand{ \mg}{\color{RedViolet}}

\newcommand{\ep}{\varepsilon}
\newcommand{\rr}{{\mathbb R}}
\newcommand{\alert}[1]{\fbox{#1}}

\newcommand{\eqd}{\sim}
\def\R{{\mathbb R}}
\def\N{{\mathbb N}}
\def\Q{{\mathbb Q}}
\def\C{{\mathbb C}}
\def\l{{\langle}}
\def\r{\rangle}
\def\t{\tau}
\def\k{\kappa}
\def\a{\alpha}
\def\la{\lambda}
\def\De{\Delta}
\def\de{\delta}
\def\ga{\gamma}
\def\Ga{\Gamma}
\def\ep{\varepsilon}
\def\eps{\varepsilon}
\def\si{\sigma}
\def\Re {{\rm Re}\,}
\def\Im {{\rm Im}\,}
\def\E{{\mathbb E}}
\def\P{{\mathbb P}}
\def\Z{{\mathbb Z}}
\def\D{{\mathbb D}}
\newcommand{\ceil}[1]{\lceil{#1}\rceil}

\title{Existence of Traveling wave solutions to parabolic-elliptic-elliptic chemotaxis systems with logistic source}

\author{
Rachidi B. Salako and Wenxian Shen  \\
Department of Mathematics and Statistics\\
Auburn University\\
Auburn University, AL 36849\\
U.S.A. }

\date{}
\maketitle
\begin{abstract}
\noindent The current paper is devoted to the study of traveling wave solutions of the following parabolic-elliptic-elliptic chemotaxis systems,
\begin{equation}\label{main-eq-abstract}
\begin{cases}
u_{t}= \Delta u-\chi_1  \nabla \cdot (u \nabla v_1)+\chi_2  \nabla \cdot (u \nabla v_2) + u(a-bu),\quad x\in\mathbb{R}^N\\
0=\Delta v_1-\lambda_1v_1+\mu_1u, \quad x\in\mathbb{R}^N,\\
0=\Delta v_2-\lambda_2v_2+\mu_2u, \quad x\in\mathbb{R}^N,
\end{cases}
\end{equation}
where $a>0,\ b>0,$ $u(x,t)$ represents the population density of a mobile species, $v_1(x,t), $ represents the population density of a  chemoattractant, $v_2(x,t)$ represents the population density of a  chemo-repulsion, $\chi_1\geq 0$ and $\chi_2\geq 0$ represent the chemotaxis sensitivities, and the positive constants $\lambda_1,\lambda_2,\mu_1$, and $\mu_2$ are related to growth rate of the chemical substances. It was proved in an earlier work by the authors of the current paper that there is a nonnegative constant $K$ depending on the parameters $\chi_1,\mu_1,\lambda_1,\chi_2,\mu_2$, and $\lambda_2$ such that if $b+\chi_2\mu_2>\chi_1\mu_1+K$, then the positive constant steady solution $(\frac{a}{b},\frac{a\mu_1}{b\lambda_1},\frac{a\mu_2}{b\lambda_2})$ of \eqref{main-eq-abstract} is asymptotically stable with respect to positive perturbations. In the current paper, we prove that if $b+\chi_2\mu_2>\chi_1\mu_1+K$, then there exist a positive number $c^{*}(\chi_1,\mu_1,\lambda_1,\chi_2,\mu_2,\lambda_2)\geq 2\sqrt{a}$ such that for every $ c\in ( c^{*}(\chi_1,\mu_1,\lambda_1,\chi_2,\mu_2,\lambda_2)\ ,\ \infty)$ and $\xi\in S^{N-1}$, the system has a traveling wave solution $(u,v_1,v_2)=(U(x\cdot\xi-ct),V_1(x\cdot\xi-ct),V_2(x\cdot\xi-ct))$ with speed $c$ connecting the constant solutions $(\frac{a}{b},\frac{a\mu_1}{b\lambda_1},\frac{a\mu_2}{b\lambda_2})$ and $(0,0,0)$, and it does not have such traveling wave solutions of speed less than $2\sqrt a$. Moreover we prove that
\begin{equation*}
\lim_{(\chi_{1},\chi_2)\to (0^+,0^+)}c^{*}(\chi_1,\mu_1,\lambda_1,\chi_2,\mu_2,\lambda_2)=\begin{cases}
2\sqrt{a}\quad \text{if}\quad  a\leq \min\{\lambda_1, \lambda_2\}\\
\frac{a+\lambda_1}{\sqrt{\lambda_1}}\quad \text{if}\quad  \lambda_1\leq \min\{a, \lambda_2\}\\
\frac{a+\lambda_2}{\sqrt{\lambda_2}}\quad \text{if}\quad  \lambda_2\leq \min\{a, \lambda_1\}
\end{cases} \forall \lambda_1,\lambda_2,\mu_1,\mu_2>0,
\end{equation*}
and
\begin{equation*}
\lim_{x\to \infty}\frac{U(x)}{e^{-\sqrt a \mu x}}=1,
\end{equation*}
where $\mu$ is the only solution of the equation $\sqrt{a}(\mu+\frac{1}{\mu})=c$ in the interval $(0\ ,\ \min\{1,\sqrt{\frac{\lambda_1}{a}},\sqrt{\frac{\lambda_2}{a}}\})$.
\end{abstract}



\section{Introduction and Statement of the Main Results}

Chemotaxis is the ability for micro-organisms to respond to chemical signals by moving along gradient of  chemical substances, either toward the higher concentration (positive taxis)   or away from it (negative taxis). The following system of partial differential equations  describes the time evolution of both the density $u(x,t)$ of a mobile species and the density $v(x,t)$ of a  chemical substance,
\begin{equation}\label{IntroEq0}
\begin{cases}
u_{t}=\nabla\cdot (\nabla u- \chi(u)\nabla v) + u(a-bu),\quad   x\in\Omega \\
\tau v_t=\Delta v -v+u,\quad  x\in\Omega,
\end{cases}
\end{equation}
 complemented with certain boundary condition on $\partial\Omega$ if $\Omega$ is bounded, where $\Omega\subset \R^N$ is an open domain;  $\tau\ge 0$ is a non-negative constant linked to the speed of diffusion of the chemical; $a$ and $b$ are nonnegative constant real numbers related to the growth rate of the mobile species;  and the function $\chi(u)$ represents the sensitivity with respect to chemotaxis.  System \eqref{IntroEq0}  has attracted a number of researchers over the last three decades.  It is a simplified version of the mathematical model of chemotaxis (aggregation of organisms sensitive to a gradient of a chemical
substance) proposed by Keller and Segel  (see \cite{KeSe1}, \cite{KeSe2}).
In literature, \eqref{IntroEq0} is called the Keller-Segel model or a chemotaxis model.
The nature of the sensitivity function of the mobile species with respect to the chemical signal, $\chi(u)$, plays important role in the features of solutions of \eqref{IntroEq0}. In the case of positive sensitivity function $\chi(u)$, the mobile species moves toward the higher concentration of the chemical substances and \eqref{IntroEq0} is referred to as chemoattraction models. If $\chi(u)$ has negative sign, the mobile species moves away the higher concentration of the chemical substances and \eqref{IntroEq0} is referred to as chemorepulsion models.

 It is well known that chemotactic cross-diffusion is somewhat "dangerous" in the sense that finite-time blow-up might occur.
 For example, for the choice of $\chi(u)=\chi u, \chi>0,$ and no logistic function, i.e $a=b=0$, finite time blow-up may occur in \eqref{IntroEq0} (see \cite{BBTW, DiNaRa, NAGAI_SENBA_YOSHIDA} for the case $\tau=0$ and \cite{HerVel,win_arxiv} for the case $\tau=1$), but that this situation is less dangerous when this taxis is repulsive (see \cite{CiLaCr} and the references therein). When the logistic source is not identically zero, that is $a>0$ and $b>0$, the blow-up phenomena in the chemoattraction case may be suppressed to some extent,  namely, finite time blow up does not occur
  in \eqref{IntroEq0} with $\chi(u)=\chi u$  provided that the logistic damping constant $b$ is sufficiently large relative to $\chi$ (see \cite{TeWi} for $\tau=0$ and \cite{Win} for $\tau=1$). Quite rich dynamical features may be observed in such systems,   either numerically  or also analytically.
  But many fundamental dynamical issues for \eqref{IntroEq0} are not well understood yet, in particular, when $\Omega$ is unbounded.

Since the works by Keller and Segel,  a rich variety of mathematical models for studying chemotaxis have appeared (see \cite{BBTW}, \cite{DiNa}, \cite{DiNaRa}, \cite{GaSaTe}, \cite{KKAS},  \cite{NAGAI_SENBA_YOSHIDA},\cite{SaSh1}, \cite{Sug},  \cite{SuKu},  \cite{TeWi},   \cite{WaMuZh},  \cite{win_jde}, \cite{win_JMAA_veryweak}, \cite{win_arxiv}, \cite{Win}, \cite{win_JNLS}, \cite{YoYo}, \cite{ZhMuHuTi}, and the references therein). In the current paper,  we consider chemoattraction-repulsion process in which cells undergo random motion and chemotaxis towards attractant and away from repellent \cite{MLACLE}. More precisely, we consider the model with proliferation and death of cells and assume that chemicals diffuse very quickly. These lead to the model of partial differential equations as follows:
\begin{equation}\label{Main-eq0}
\begin{cases}
u_{t}=\Delta u -\chi_1\nabla( u\nabla v_1)+\chi_2 \nabla(u\nabla v_2 )+ u(a -b u) \qquad \  x\in\Omega,\ t>0,  \\
0=(\Delta- \lambda_1 I)v_1+ \mu_1 u  \qquad \ x\in\Omega,\ t>0,  \\
0=(\Delta- \lambda_2 I)v_2+ \mu_2 u  \qquad \ x\in\Omega,\ t>0
\end{cases}
\end{equation}
complemented with certain boundary condition on $\partial\Omega$ if $\Omega$ is bounded. It is clear that that if either $\chi_1=0$ or $\chi_2=0$ then \eqref{Main-eq0} becomes  \eqref{IntroEq0} with $\chi(u)=\chi u$.   As \eqref{IntroEq0}, among the central problems about \eqref{Main-eq0} are global existence of classical/weak solutions with given initial functions; finite-time blow-up; pattern formation; existence, uniqueness, and stability of certain special solutions; spatial spreading and front propagation dynamics when the domain is a whole space; etc.
 Similarly, many of these central problems are not well understood yet, in particular, when $\Omega$ is unbounded.

Global existence and asymptotic behavior of solutions of \eqref{Main-eq0} on bounded domain $\Omega$ complemented with Neumann boundary conditions,
\begin{equation}\label{Main Intro-eq1}
\frac{\partial u}{\partial n}=\frac{\partial v_1}{\partial n}=\frac{\partial v_2}{\partial n}=0,
\end{equation}
have been studied in many papers (see \cite{EEspejoand-TSuzuki,  Horstman, Jin,Lin_Mu_Gao, JLZAW,  PLJSZW, MLACLE, TaWa, YWang1, WangXiang, QZhanYLi, PCHu1} and the references therein). For example, in \cite{QZhanYLi}, among others, the authors proved that, if $b>\chi_{1}\mu_1-\chi_2\mu_2$, or $N\leq 2$, or $\frac{N-2}{N}(\chi_1\mu_1-\chi_2\mu_2)<b$ and $N\geq 3$,
then for every nonnegative initial $u_0\in C^{0}(\overline{\Omega}),$ \eqref{Main-eq0}+\eqref{Main Intro-eq1} has a unique global classical solution $(u(\cdot,\cdot),v_1(\cdot,\cdot),v_2(\cdot,\cdot))$ which is uniformly bounded. In the case that there is no logistic function, the authors in \cite{TaWa} studied  both \eqref{Main-eq0} and its full parabolic variant.

 In a very recent work \cite{SaSh4}, the authors of the current paper studied the global existence, asymptotic behavior, and spatial spreading properties of classical solutions of \eqref{Main-eq0}  on the whole space $\Omega=\R^N$,  that is,
\begin{equation}\label{Main-eq01}
\begin{cases}
u_{t}=\Delta u -\chi_1\nabla( u\nabla v_1)+\chi_2 \nabla(u\nabla v_2 )+ u(a -b u) \qquad \  x\in\R^N,\ t>0,  \\
0=(\Delta- \lambda_1 I)v_1+ \mu_1 u  \qquad \ x\in\R^N,\ t>0,  \\
0=(\Delta- \lambda_2 I)v_2+ \mu_2 u  \qquad \ \ x\in\R^N,\ t>0.
\end{cases}
\end{equation}

Let
\begin{equation}
\label{unif-cont-space}
C_{\rm unif}^b(\R^N)=\{u\in C(\R^N)\,|\, u(x)\,\,\text{is uniformly continuous in}\,\, x\in\R^N\,\, {\rm and}\,\, \sup_{x\in\R^N}|u(x)|<\infty\}
\end{equation}
equipped with the norm $\|u\|_\infty=\sup_{x\in\R}|u(x)|$. For every real number $r$, we let $(r)_+=\max\{0,r\}$ and $(r)_-=\max\{0,-r\}$. 
Let
\begin{equation}\label{const-upper-bound}
M:=\min\Big\{ \frac{(\chi_2\mu_2\lambda_2-\chi_1\mu_1\lambda_1)_++\chi_1\mu_1(\lambda_1-\lambda_2)_+}{\lambda_2} , \frac{\chi_2\mu_2(\lambda_1-\lambda_2)_+ + (\chi_2\mu_2\lambda_2-\chi_1\mu_1\lambda_1)_+}{\lambda_1}  \Big\}
\end{equation}
and
\begin{equation}\label{stability-cond}
K:=\min\Big\{\frac{|\chi_2\lambda_2\mu_2-\chi_1\lambda_1\mu_1|+\chi_1\mu_1|\lambda_1-\lambda_2|}{\lambda_2},\frac{ \chi_2\mu_2|\lambda_1-\lambda_2| + |\chi_2\mu_2\lambda_2-\chi_1\mu_1\lambda_1|}{\lambda_{1}}   \Big\}.
\end{equation}
Among others, the following are proved in \cite{SaSh4}.

\smallskip

\noindent$\bullet$ {\it If $b+\chi_2\mu_2>\chi_1\mu_1+M$, then  for every nonnegative initial function $u_0\in C^{b}_{\rm unif}(\R^{N})$, \eqref{Main-eq01} has a unique bounded and globally defined classical solution $(u(\cdot,\cdot;u_0),(v_{1}(\cdot,\cdot;u_0)),v_{2}(\cdot,\cdot;u_0))$
 with $u(\cdot,0;u_0)=u_0$ and  $\|u(\cdot,t;u_0)\|_{\infty}\leq \max\{\|u_0\|_{\infty}\ ,\ \frac{a}{b+\chi_2\mu_2-\chi_1\mu_1-M}\}$, for all $ t\geq 0. $
}

\smallskip

\noindent$\bullet$ {\it  If $b+\chi_2\mu_2>\chi_1\mu_1+K$, then  for every initial function $u_0\in C^{b}_{\rm unif}(\R^{N})$ with $\inf_{x\in\R^N}u_{0}(x)>0$, the unique bounded and globally defined classical solution $(u(\cdot,\cdot;u_0),(v_{1}(\cdot,\cdot;u_0)),v_{2}(\cdot,\cdot;u_0))$ \eqref{Main-eq01} with $u(\cdot,0;u_0)=u_0$  satisfies $$\lim_{t\to \infty}\Big[\|u(\cdot,t;u_0)-\frac{a}{b}\|_{\infty}+\|v_1(\cdot,t;u_0)-\frac{a\mu_1}{b\lambda_1}\|_{\infty}+ \|v_2(\cdot,t;u_0)-\frac{a\mu_2}{b\lambda_2}\|_{\infty}\Big]=0.$$ That is the constant steady state $(\frac{a}{b},\frac{a\mu_1}{b\lambda_1},\frac{a\mu_2}{b\lambda_2})$ is asymptotically stable with respect to strictly  positive perturbations if $b+\chi_2\mu_2>\chi_1\mu_1+K$.
}

\smallskip

It  is not yet known whether it is enough for $b+\chi_2\mu_2>\chi_1\mu_1+M$ to guarantee the stability of the positive constant steady solution  with respect to strictly positive perturbations.

Note that, in the absence of the chemotaxis (i.e. $\chi_1=\chi_2=0$ or $\chi_1-\chi_2= \mu_1-\mu_2=\lambda_1-\lambda_2=0$), the first equation of \eqref{Main-eq0} with $\Omega=\R^N$ becomes
\begin{equation}
\label{kpp-eq}
u_{t}=\Delta u+ u(a-bu),\quad  x\in\R^N,
\end{equation}
which is referred to as Fisher or KPP equation  due to  the pioneering works by Fisher (\cite{Fis}) and Kolmogorov, Petrowsky, Piscunov
(\cite{KPP}).  Among important solutions of \eqref{kpp-eq}
are traveling wave solutions of \eqref{kpp-eq} connecting the constant solutions $a/b$ and $0$.
It is well known that \eqref{kpp-eq} has traveling wave solutions $u(t,x)=\phi(x-ct)$ connecting $\frac{a}{b}$ and $0$ (i.e.
$(\phi(-\infty)=a/b$, $\phi(\infty)=0)$) for all speeds $c\geq 2\sqrt{a}$ and has no such traveling wave
solutions of slower speed (see \cite{Fis, KPP, Wei1}). Moreover, the stability of traveling wave solutions of \eqref{kpp-eq} connecting $\frac{a}{b}$ and $0$ has also been studied (see \cite{Bra}, \cite{Sat}, \cite{Uch}, etc.).
The above mentioned results for \eqref{kpp-eq} have also been
well extended to   reaction diffusion equations of the form,
\begin{equation}
\label{general-kpp-eq}
u_t=\Delta u+u f(t,x,u),\quad x\in\R^N,
\end{equation}
where $f(t,x,u)<0$ for $u\gg 1$,  $\partial_u f(t,x,u)<0$ for $u\ge 0$ (see \cite{Berestycki1, BeHaNa1, BeHaNa2, Henri1, Fre, FrGa, LiZh, LiZh1, Nad, NoRuXi, NoXi1, She1, She2, Wei1, Wei2, Zla}, etc.).

Similar to \eqref{kpp-eq}, traveling wave solutions  connecting the constant solutions $(\frac{a}{b},\frac{a\mu_1}{b\lambda_1},\frac{a\mu_2}{b\lambda_2})$ and $(0,0,0)$ are among most important solutions of \eqref{Main-eq01}. However, such solutions have been hardly studied. The objective of the current paper is to study the existence of traveling wave solutions connecting $(\frac{a}{b},\frac{a\mu_1}{b\lambda_1},\frac{a\mu_2}{b\lambda_2})$ and $(0,0,0)$.
A nonnegative solution $(u(x,t),v_1(x,t),v_2(x,t))$ of \eqref{Main-eq01} is called a {\it traveling wave solution} connecting $(\frac{a}{b},\frac{a\mu_1}{b\lambda_1},\frac{a\mu_2}{b\lambda_2})$ and $(0,0,0)$ and propagating in the direction $\xi\in S^{N-1}$ with speed $c$ if it is of the form
$(u(x,t),v_1(x,t),v_2(x,t))=(U(x\cdot\xi-ct),V_1(x\cdot\xi-ct),V_2(x\cdot\xi-ct))$ with
$\lim_{z\to -\infty}(U(z),V_1(z),V_2(z))=(\frac{a}{b},\frac{a\mu_1}{b\lambda_1},\frac{a\mu_2}{b\lambda_2})$ and $\lim_{z\to\infty}(U(z),V_1(z),V_2)=(0,0,0)$.

Observe that, if $(u(x,t),v_1(x,t),v_2(x,t))=(U(x\cdot\xi-ct),V_1(x\cdot\xi-ct),V_2(x\cdot\xi-ct))$  $(x\in\R^N, \ t\ge 0)$ is a traveling wave solution of \eqref{Main-eq01} connecting  $(\frac{a}{b},\frac{a\mu_1}{b\lambda_1},\frac{a\mu_2}{b\lambda_2})$ and $(0,0,0)$ and propagating
in the direction $\xi\in S^{N-1}$, then $(u,v_1,v_2)=(U(x-ct),V_1(x-ct),V_2(x-ct))$ ($x\in\R$)
is a traveling wave solution of
\begin{equation}
\label{Main-eq1}
\begin{cases}
\partial_tu=\partial_{xx}u +\partial_x(u \partial_x(\chi_2v_2-\chi_1v_1)) + u(a-bu),\quad x\in\R\\
0=\partial_{xx}v_{1}-\lambda_1v_1+ \mu_1u, \quad x\in\R\\
0=\partial_{xx}v_{2}-\lambda_2v_2+ \mu_2u, \quad x\in\R
\end{cases}
\end{equation}
connecting  $(\frac{a}{b},\frac{a\mu_1}{b\lambda_1},\frac{a\mu_2}{b\lambda_2})$ and $(0,0,0)$. Conversely, if $(u(x,t),v_1(x,t),v_2(x,t))=(U(x-ct),V_1(x-ct),V_2(x,t))$ ($x\in\R,\ t\geq 0$) is a traveling wave solution
of \eqref{Main-eq1} connecting $(\frac{a}{b},\frac{a\mu_1}{b\lambda_1},\frac{a\mu_2}{b\lambda_2})$ and $(0,0,0)$, then $(u,v_1,v_2)=(U(x\cdot\xi-ct),V_1(x\cdot\xi-ct),V_{2}(x\cdot\xi-ct))$  $(x\in\R^N)$ is a traveling wave solution of \eqref{Main-eq01} connecting  $(\frac{a}{b},\frac{a\mu_1}{b\lambda_1},\frac{a\mu_2}{b\lambda_2})$ and $(0,0,0)$ and propagating in the direction $\xi\in S^{N-1}$. In the following, we will then study the existence of traveling wave solutions
of \eqref{Main-eq1} connecting $(\frac{a}{b},\frac{a\mu_1}{b\lambda_1},\frac{a\mu_2}{b\lambda_2})$ and $(0,0,0)$.

Observe also that $(u,v_1,v_2)=(U(x-ct),V_1(x-ct),V_2(x-ct))$  is a traveling wave solution of \eqref{Main-eq1} connecting $(\frac{a}{b},\frac{a\mu_1}{b\lambda_1},\frac{a\mu_2}{b\lambda_2})$ and $(0,0,0)$ with speed $c$ if and only if $(u,v_1,v_2)=(U(x),V_1(x),V_2(x))$ is a stationary solution of
the following parabolic-elliptic-elliptic chemotaxis system,
\begin{equation}
\label{Main-eq2}
\begin{cases}
\partial_tu=\partial_{xx}u +c\partial_{x}u +\partial_x(u \partial_x(\chi_2v_2-\chi_1v_1)) + u(a-bu),\quad x\in\R\\
0=\partial_{xx}v_{1}-\lambda_1v_1+\mu_1u, \quad x\in\R\\
0=\partial_{xx}v_{2}-\lambda_2v_2+\mu_2u, \quad x\in\R
\end{cases}
\end{equation}
 connecting  $(\frac{a}{b},\frac{a\mu_1}{b\lambda_1},\frac{a\mu_2}{b\lambda_2})$ and $(0,0,0)$.
In this paper, to study the existence of traveling
wave solutions of \eqref{Main-eq1}, we study the existence of constant $c$'s so that \eqref{Main-eq2} has a stationary solution $(U(x),V_1(x),V_2(x))$ satisfying
$(U(-\infty),V_1(-\infty),V_2(-\infty))=(\frac{a}{b},\frac{a\mu_1}{b\lambda_1},\frac{a\mu_2}{b\lambda_2})$ and $(U(\infty),V_1(\infty)$, $V_2(\infty))=(0,0,0)$.

Throughout this paper we shall always suppose that
\begin{equation}\label{H1}
b+\chi_2\mu_2>\chi_1\mu_1+M,
\end{equation}
where $M$ is as in \eqref{const-upper-bound}.
We prove the following theorems on the existence and nonexistence of traveling wave solutions of \eqref{Main-eq2}.

\medskip

\noindent{\bf Theorem A.}
{\it  If $\ b+\chi_2\mu_2>\chi_1\mu_1+K$, then there exist a positive number $c^{*}(\chi_1,\mu_1,\lambda_1,\chi_2,\mu_2,\lambda_2)>0$ such that for every $ c\in  ( c^{*}(\chi_1,\mu_1,\lambda_1,\chi_2,\mu_2,\lambda_2)\ ,\ \infty)$, \eqref{Main-eq1} has a traveling wave solution $(u(x,t),v_1(x,t),v_2(x,t))=(U(x-ct),V_1(x-ct),V_2(x-ct))$  connecting the constant solutions $(\frac{a}{b},\frac{a\mu_1}{b\lambda_1},\frac{a\mu_2}{b\lambda_2})$ and $(0,0,0)$ with speed $c$. Moreover,
\begin{equation}
\lim_{(\chi_{1},\chi_2)\to (0^+,0^+)}c^{*}(\chi_1,\mu_1,\lambda_1,\chi_2,\mu_2,\lambda_2)=\begin{cases}
2\sqrt{a}\quad \text{if}\quad  a\leq \min\{\lambda_1, \lambda_2\}\\
\frac{a+\lambda_1}{\sqrt{\lambda_1}}\quad \text{if}\quad  \lambda_1\leq \min\{a, \lambda_2\}\\
\frac{a+\lambda_2}{\sqrt{\lambda_2}}\quad \text{if}\quad  \lambda_2\leq \min\{a, \lambda_1\}
\end{cases},\quad \forall \lambda_1,\lambda_2,\mu_1,\mu_2>0,
\end{equation}
and
\begin{equation}
\label{tail-eq}
\lim_{x\to \infty}\frac{U(x)}{e^{-\sqrt a \mu x}}=1,
\end{equation}
where $\mu$ is the only solution of the equation $\sqrt{a}(\mu+\frac{1}{\mu})=c$ in the interval $(0\ ,\ \min\{1,\sqrt{\frac{\lambda_1}{a}},\sqrt{\frac{\lambda_2}{a}}\})$.
}
\medskip

We will prove Theorem A by first constructing proper sub-solutions and sup-solutions of a collection of parabolic equations, a non-empty bounded and convex subset $\mathcal{E}_{\mu}$ of $C_{\rm unif}^b(\R)$, and a mapping from $\mathcal{E}_{\mu}$, and then showing that the mapping has a fixed point,
which gives rise to a traveling wave solution of \eqref{Main-eq1} satisfying \eqref{tail-eq}.

\medskip

\noindent {\bf Remark B.}
{\it Suppose that $\lambda_1=\lambda_2=\lambda>0$. Then,
\begin{description}
\item[(1)] For every $a>0$, $b>0$, $\mu_1>0$, $\mu_2>0$, and $\chi_1\geq0$, $\chi_2\geq0$, if $\chi_{1}\mu_1=\chi_2\mu_2$, then $c^*(\chi_1,\mu_1,\lambda,\chi_2,\mu_2,\lambda)=\begin{cases}
2\sqrt{a}\quad \text{if} \quad a\leq \lambda\\
\frac{a+\lambda}{\sqrt{\lambda}}\quad \text{if} \quad a\geq \lambda.
\end{cases}$
\item[(2)] For every $a>0$, $b>0$, $\mu_1>0$, $\mu_2>0$, and $\chi_1\geq0$, $\chi_2\geq0$, if $0<\chi_{1}\mu_1-\chi_2\mu_2<\frac{b}{2}$, then   $\lim_{\chi_1\mu_1-\chi_2\mu_2\to 0^{+}}c^*(\chi_1,\mu_1,\lambda,\chi_2,\mu_2,\lambda)=\begin{cases}
2\sqrt{a}\quad \text{if} \quad a\leq \lambda\\
\frac{a+\lambda}{\sqrt{\lambda}}\quad \text{if} \quad a\geq \lambda.
\end{cases}$
\item[(3)] For every $a>0$, $b>0$, $\mu_1>0$, $\mu_2>0$, and $\chi_1\geq0$, $\chi_2\geq0$, if $\chi_{1}\mu_1<\chi_2\mu_2$, then   $\lim_{\chi_2\mu_2-\chi_1\mu_1\to 0^{+}}c^*(\chi_1,\mu_1,\lambda,\chi_2,\mu_2,\lambda)=\begin{cases}
2\sqrt{a}\quad \text{if} \quad a\leq \lambda\\
\frac{a+\lambda}{\sqrt{\lambda}}\quad \text{if} \quad a\geq \lambda.
\end{cases}$
\end{description}
}

\medskip

\noindent {\bf Theorem C.} {\it
For any given $\chi_{i}\ge 0$, $\lambda_i>0$  and $\mu_i>0$ $(i=1,2)$, \eqref{Main-eq1} has no traveling wave solution $(u,v_1,v_2)=(U(x-ct),V_1(x-ct),V_2(x-ct))$
with $(U(-\infty),V_1(-\infty),V_2(-\infty))=(\frac{a}{b},\frac{a\mu_1}{b\lambda_1},\frac{a\mu_2}{b\lambda_2})$, $(U(\infty),V_1(\infty),V_2(\infty))=(0,0,0)$, and $c<2 \sqrt a$.
}

\medskip 

\smallskip

\noindent{\bf Remark D.}{\it
\begin{description}
\item[(i)] It follows from Theorem A that if either $(\chi_{1},\chi_2)\to (0^+,0^+)$ or $(\chi_{1}-\chi_2,\mu_1-\mu_2,\lambda_1-\lambda_2)\to (0,0,0)$, then $c^{*}(\chi_1,\mu_1,\lambda_1,\chi_2,\mu_2,\lambda_2)$ converges to the minimal speed of \eqref{kpp-eq}.

\item[(ii)] When $\chi_2=0$ in Theorem A, we recover Theorem A in \cite{SaSh2}

 \item[(iii)]  For given  $\chi_{i}\ge 0$, $\lambda_{i}$, $\mu_{i}>0$ $(i=1,2)$,  let $c_{\rm min}^{*}(\chi_1,\mu_1,\lambda_1$, $\chi_2,\mu_2,\lambda_2)$ be such that
 for any  $c\geq c_{\rm min}^{*}(\chi_1,\mu_1,\lambda_1,\chi_2,\mu_2,\lambda_2)$, \eqref{Main-eq1} has a traveling wave solution connecting $(\frac{a}{b},\frac{a\mu_1}{b\lambda_1},\frac{a\mu_2}{b\lambda_2})$ and $(0,0,0)$ with speed $c$ and \eqref{Main-eq1} has no traveling solution for speed less than $c_{\rm min}^{*}(\chi_1,\mu_1,\lambda_1,\chi_2,\mu_2,\lambda_2)$.
 By Theorems A and B, we have $$2\sqrt{a}\leq c_{\rm min}^{*}(\chi_1,\mu_1,\lambda_1,\chi_2,\mu_2,\lambda_2). $$
 It remains  open whether $c_{\rm min}^{*}(\chi_1,\mu_1,\lambda_1,\chi_2,\mu_2,\lambda_2)=2\sqrt a$. This question is about whether the chemotaxis increases the minimal wave speed of the existence of traveling wave
 solutions. It is of great theoretical and biological interests to investigate this question.
\end{description}
}



It should be  pointed out that there are many studies on traveling wave solutions of several other types of chemotaxis models, see, for example, \cite{AiHuWa, AiWa, FuMiTs, HoSt, LiLiWa,
MaNoSh, NaPeRy,Wan, SaSh2, SaSh3}, etc. In particular, the reader is referred to
the review paper \cite{Wan}.

The rest of this paper is organized as follows. Section 2  is to establish  the tools that will be needed in the proof of Theorem A. It is here that we define the two special functions, which are sub-solution and sup-solution of a collection of parabolic equations,  and a non-empty bounded and convex subset $\mathcal{E}_{\mu}$ of $C_{\rm unif}^b(\R)$. In section 3, we study  the existence and nonexistence of traveling wave solutions and prove Theorems A and B.

\section{Super- and sub-solutions}

In this section, we will construct super- and sub-solutions of some related equations of \eqref{Main-eq2}   and a non-empty bounded and convex subset $\mathcal{E}_{\mu}$ of $C_{\rm unif}^b(\R)$, which will be used to prove the existence of traveling wave solutions of \eqref{Main-eq1} in next section. Throughout this section we suppose that $a>0$  and $b>0$ are given positive real numbers.


For given $0<\nu<1$, let
\begin{equation}
\label{holder-cont-space}
C^{\nu}_{\rm unif}(\R^N)=\{u\in C_{\rm unif}^b(\R^N)\,|\, \sup_{x,y\in\R^N,x\not =y}\frac{|u(x)-u(y)|}{|x-y|^\nu}<\infty\}
\end{equation}
equipped with the norm $\|u\|_{C^\nu_{\rm unif}}=\sup_{x\in\R^N}|u(x)|+\sup_{x,y\in\R^N,x\not =y}\frac{|u(x)-u(y)|}{|x-y|^\nu}$.

For $0<\mu <\min\{1,\sqrt{\frac{\lambda_1}{a}}, \sqrt{\frac{\lambda_1}{a}} \}$, define
$$c_{\mu}=\sqrt{a}(\mu+\frac{1}{\mu})\quad {\rm and}\quad \varphi_{\mu}(x)=e^{-\sqrt{a}\mu x}\quad \forall\,\, x\in\R.$$
Note that for every fixed
$0<\mu  < \min\{1, \sqrt{\frac{\lambda_1}{a}}, \sqrt{\frac{\lambda_2}{a}}\}$, we have that $1- \mu^2 >0$; $a\mu^2-\lambda_i>0$ $(i=1,2)$;  the function $\varphi_{\mu}$ is decreasing, infinitely many differentiable, and satisfies
 \begin{equation}\label{Eq1 of varphi}
\varphi_{\mu}''(x)+c_{\mu}\varphi_{\mu}'(x)+a\varphi_{\mu}(x)=0 \quad\forall\ x\in\R
\end{equation}
and
\begin{equation}\label{Eq2 of varphi}
\frac{\mu_i}{a\mu^2-\lambda_i}\varphi_{\mu}''(x)-\frac{\mu_i}{a\mu^2-\lambda_i}(\lambda_i\varphi _{\mu})(x)=\mu_i\varphi_{\mu}(x)\quad (i=1,2).
\end{equation}

For every $0<\mu<\min\{1, \sqrt{\frac{\lambda_1}{a}}, \sqrt{\frac{\lambda_2}{a}}\}$,    define
\begin{align}\label{u-upper}
U_{\mu}^{+}(x)&=\min\{\frac{a}{b+\chi_2\mu_2-\chi_1\mu_1-M}, \varphi_{\mu}(x)\}\nonumber\\
&=\begin{cases}
\frac{a}{b+\chi_2\mu_2-\chi_1\mu_1-M} \ \quad \text{if }\ x\leq \frac{-\ln(\frac{a}{b+\chi_2\mu_2-\chi_1\mu_1-M})}{\mu}\\
e^{- \sqrt a \mu x} \qquad\qquad\qquad \text{if}\ x\geq \frac{-\ln(\frac{a}{b+\chi_2\mu_2-\chi_1\mu_1-M})}{\mu}.
\end{cases}
\end{align}
Since $\varphi_{\mu}$ is decreasing, then the function $U^{+}_{\mu}$ is non-increasing. Furthermore, the function $U^{+}_{\mu}$ belongs to $ C^{\nu}_{\rm unif}(\R)$ for every $0\leq \nu< 1$,  $0<\mu<\min\{1, \sqrt{\frac{\lambda_1}{a}}, \sqrt{\frac{\lambda_2}{a}}\}$.

Let  $0<\mu<\min\{1, \sqrt{\frac{\lambda_1}{a}}, \sqrt{\frac{\lambda_2}{a}}\}$ be fixed. Next, let $\mu<\tilde{\mu}<\min\{1, 2\mu,\sqrt{\frac{\lambda_1}{a}}, \sqrt{\frac{\lambda_2}{a}}\}$ and $d>1$. The function $\varphi_{\mu}-d\varphi_{\tilde{\mu}}$ achieved its maximum value at $\bar{a}_{\mu,\tilde{\mu},d}:=\frac{\ln(d\tilde{\mu})-\ln(\mu)}{(\tilde{\mu}-\mu)\sqrt{a}}$ and takes the value zero at $\underline{a}_{\mu,\tilde{\mu},d}:= \frac{\ln(d)}{(\tilde{\mu}-\mu)\sqrt{a}}$.
Define
\begin{equation}\label{u-lower}
U_{\mu}^{-}(x):= \max\{ 0, \varphi_{\mu}(x)-d\varphi_{\tilde{\mu}}(x)\}=\begin{cases}
0\qquad \qquad \qquad \quad \text{if}\ \ x\leq \underline{a}_{\mu,\tilde{\mu},d}\\
\varphi_{\mu}(x)-d\varphi_{\tilde{\mu}}(x)\quad \text{if}\ x\geq \underline{a}_{\mu,\tilde{\mu},d}.
\end{cases}
\end{equation}
It is clear that
$$0\leq U_{\mu}^{-}(x)\leq U^{+}_{\mu}(x)\leq \frac{a}{b+\chi_2\mu_2-\chi_1\mu_1-M}$$
for every $x\in\R$, and $U_{\mu}^{-}\in  C^{\nu}_{\rm unif}(\R)$ for every $0\leq \nu< 1$. Finally, let  the set $\mathcal{E}_{\mu}$ be defined by
\begin{equation}\label{definition-E-mu}
\mathcal{E}_{\mu}=\{u\in C^{b}_{\rm unif}(\R) \,|\, U_{\mu}^{-}\leq u\leq U_{\mu}^{+}\}.
\end{equation}
It should be noted that $U_{\mu}^{-}$ and $\mathcal{E}_{\mu}$ all depend on $\tilde{\mu}$ and $d$. Later on, we shall provide more information on how to choose $d$ and $\tilde{\mu}$ whenever $\mu$ is given.

For clarity, we introduce the following quantities, which will also be useful in the statement of our main results in this section as well for the subsequent sections,
\begin{equation}\label{D}\tilde K=\min\Big\{\frac{|\chi_1\mu_1-\chi_2\mu_2|}{\sqrt{\lambda_2}}+\frac{\chi_1\mu_1|\sqrt{\lambda_1}-\sqrt{\lambda_2}|}{\sqrt{\lambda_1\lambda_2}}\ ,\ \frac{|\chi_1\mu_1-\chi_2\mu_2|}{\sqrt{\lambda_1}}+\frac{\chi_2\mu_2|\sqrt{\lambda_1}-\sqrt{\lambda_2}|}{\sqrt{\lambda_1\lambda_2}} \Big\},
\end{equation}
\begin{align}\label{const-lower-bound}
\tilde M=\min\Big\{ \frac{(\chi_2\lambda_2\mu_2-\chi_1\lambda_1\mu_1)_{-}+\chi_1\mu_1(\lambda_1-\lambda_2)_{-}}{\lambda_2}, \frac{\chi_2\mu_2(\lambda_1-\lambda_2)_{-} + (\chi_2\mu_2\lambda_2-\chi_1\mu_1\lambda_1)_{-}}{\lambda_{1}} \Big\},
\end{align}
\begin{align}\label{L-upper}
\overline{L}_{\mu}=\min\Big\{&\frac{\chi_1\mu_1\lambda_1(\lambda_1-\lambda_2)_+}{(\lambda_2-a\mu^2)(\lambda_1-a\mu^2)}
+\frac{(\chi_2\mu_2\lambda_2-\chi_1\mu_1\lambda_1)_{+}}{\lambda_2-a\mu^2}, \nonumber\\
&\frac{\chi_2\mu_2\lambda_2(\lambda_1-\lambda_2)_+}{(\lambda_2-a\mu^2)(\lambda_1-a\mu^2)}
+\frac{(\chi_2\mu_2\lambda_2-\chi_1\mu_1\lambda_1)_{+}}{\lambda_1-a\mu^2}\Big\},
\end{align}
\begin{align}\label{L-lower}
\underline{L}_{\mu}=\min\Big\{& \frac{\chi_1\mu_1\lambda_1(\lambda_1-\lambda_2)_{-}}{(\lambda_2-a\mu^2)(\lambda_1-a\mu^2)} + \frac{(\chi_2\mu_2\lambda_2-\chi_1\mu_1\lambda_1)_{-}}{\lambda_2-a\mu^2},\nonumber\\
& \frac{\chi_2\mu_2\lambda_2(\lambda_1-\lambda_2)_{-}}{(\lambda_2-a\mu^2)(\lambda_1-a\mu^2)} + \frac{(\chi_2\mu_2\lambda_2-\chi_1\mu_1\lambda_1)_{-}}{\lambda_1-a\mu^2}\Big\},
\end{align}
and
\begin{align}\label{K}
K_{\mu}=\min\Big\{&\frac{|\chi_2\mu_2-\chi_1\mu_1|(\sqrt{\lambda_2-a\mu^2} +\mu\sqrt{a})}{\lambda_2-a\mu^2} +\Big|\frac{\chi_1\mu_1}{\sqrt{\lambda_1-a\mu^2}}-\frac{\chi_1\mu_1}{\sqrt{\lambda_2-a\mu^2}}\Big|\nonumber\\
&+\frac{\mu\sqrt{a}\chi_1\mu_1|\lambda_1-\lambda_2|}{(\lambda_1-a\mu^2)(\lambda_2-a\mu^2)}\, ,\,  \frac{|\chi_2\mu_2-\chi_1\mu_1|(\sqrt{\lambda_1-a\mu^2} +\mu\sqrt{a})}{\lambda_1-a\mu^2} \nonumber\\
&+\Big|\frac{\chi_2\mu_2}{\sqrt{\lambda_2-a\mu^2}}-\frac{\chi_2\mu_2}{\sqrt{\lambda_1-a\mu^2}}\Big|
+\frac{\mu\sqrt{a}\chi_2\mu_2|\lambda_2-\lambda_1|}{(\lambda_2-a\mu^2)(\lambda_1-a\mu^2)}\Big\}.
\end{align}
 Observe that $M<\overline{L}_{\mu}$. Hence $b+\chi_2\mu_2-\chi_1\mu_1\geq \overline{L}_{\mu}$ implies that \eqref{H1} holds. Furthermore, we have
 \begin{equation}\label{mu-lim}
 K=M +\tilde M, \quad \lim_{\mu\to 0^{+}}\overline{L}_{\mu}=M, \quad \lim_{\mu\to 0^{+}}K_{\mu}=\tilde K, \quad \text{and}\quad \lim_{\mu\to 0^+}\mu K_{\mu}=0.
 \end{equation}

For every $u\in C_{\rm unif}^b(\R)$, consider
\begin{equation}\label{ODE2}
U_{t}=U_{xx}+(c_{\mu}+\partial_{x}(\chi_2 V_2-\chi_1 V_1)(\cdot;u))U_{x}+(a+(\chi_2\lambda_2V_2-\chi_1\lambda_1V_1)(\cdot;u)-(b+\chi_2\mu_2-\chi_1\mu_1)U)U,
\end{equation}
where
\begin{equation}\label{Inverse of u}
V_i(x;u)=\mu_i\int_{0}^{\infty}\int_{\R}\frac{e^{-\lambda_i s}}{\sqrt{4\pi s}}e^{-\frac{|x-z|^{2}}{4s}}u(z)dzds,\quad i=1,2.
\end{equation}
It is well known that the function $V_1(x;u)$ (resp. $V_2(x;u)$)
is the solution of the second equation (resp. the third equation) of \eqref{Main-eq1} in $C^{b}_{\rm unif}(\R)$ with given $u\in C_{\rm unif}^b(\R)$.

For  given open intervals $D\subset \R$ and $I\subset \R$, a function $U(\cdot,\cdot)\in C^{2,1}(D\times I,\R)$ is called a {\it super-solution} or {\it sub-solution} of \eqref{ODE2} on $D\times I$ if for every $x\in D$ and $t\in I$,
$$U_{t}\ge U_{xx}+(c_{\mu}+(\chi_2 V_2'-\chi_1 V_1')(x;u))U_{x}+(a+(\chi_2\lambda_2V_2-\chi_1\lambda_1V_1)(x;u)-(b+\chi_2\mu_2-\chi_1\mu_1)U)U,
$$
or for  every $x\in D$ and $t\in I$,
$$
U_{t}\le U_{xx}+(c_{\mu}+(\chi_2 V_2'-\chi_1 V_1')(x;u))U_{x}+(a+(\chi_2\lambda_2V_2-\chi_1\lambda_1V_1)(x;u)-(b+\chi_2\mu_2-\chi_1\mu_1)U)U.
$$

Next, we state the main result of this section.  For convenience, we introduce the following standing assumption.

\medskip
\noindent {\bf (H)}   $0<\mu<\min\{1,  \sqrt{\frac{\lambda_1}{a}}, \sqrt{\frac{\lambda_2}{a}}\}$, $M+\tilde M+\chi_1\mu_1<b+\chi_2\mu_2$, and
 \begin{equation}\label{sup-sub-solu-eq}
\mu\sqrt{a}K_{\mu}+\overline{L}_{\mu}\leq (b+\chi_2\mu_2-\chi_1\mu_1),
\end{equation}
where $M$, $\tilde M$, $K_{\mu}$ and $\overline{L}_{\mu}$ are given by \eqref{const-upper-bound}, \eqref{const-lower-bound}, \eqref{K} and \eqref{L-upper}, respectively.

\begin{tm}\label{super-sub-solu-thm}
 Assume (H).
Then  the following hold.
\begin{itemize}\item[(1)] For every $u\in \mathcal{E}_{\mu}$, we have that  $U(x,t)=\frac{a}{b+\chi_2\mu_2-\chi_1\mu_1-M}$ is  supper-solutions of \eqref{ODE2} on $\R\times\R$ where $M$ is given by \eqref{const-upper-bound}.

\item[(2)]For every $u\in \mathcal{E}_{\mu}$,  $U(x,t)=\varphi_{\mu}(x)$ is a supper-solutions of \eqref{ODE2} on $\R\times\R$.

\item[(3)] There is $d_0>1$ such  that for every $u\in \mathcal{E}_{\mu}$, we have that  $U(x,t)=U_{\mu}^-(x)$ is a sub-solution of \eqref{ODE2} on $(\underline{a}_{\mu,\tilde{\mu},d},\infty)\times \R$ for all $d\ge d_0$ and   $\mu< \tilde{\mu}<\min\{1, \mu, \sqrt{\frac{\lambda_1}{a}}, \sqrt{\frac{\lambda_2}{a}}\},$ and $K_{\mu}(\tilde{\mu}-\mu)\leq \overline{L}_{\mu}+\underline{L}_{\mu}$, where $\overline{L}_{\mu}, \underline{L}_{\mu}$ and $K_\mu$ are given by \eqref{L-upper}, \eqref{L-lower} and \eqref{K}.

\item[(4)] For every $u\in \mathcal{E}_{\mu}$, $U(x,t)=U_{\mu}^-(x_\delta)$ is a sub-solution of \eqref{ODE2} on $\R\times \R$ for $0<\delta\ll 1$, where $x_\delta=\underline{a}_{\mu,\tilde{\mu},d}+\delta$.\end{itemize}
\end{tm}

To prove Theorem \ref{super-sub-solu-thm}, we first establish some estimates on $(\chi_2\lambda_2V_2-\chi_1\lambda_1V_{1})(\cdot;u)$ and $\partial_{x}(\chi_2V_{2}-\chi_1V_1)(\cdot;u)$ for every $u\in\mathcal{E}_{\mu}$.

It follows from \eqref{Inverse of u}, that \begin{equation}\label{Estimates on Inverse of u} \max\{\|V_{i}(\cdot;u)\|_{\infty}, \ \|\partial_{x}V_{i}(\cdot;u)\|_{\infty} \}\leq \frac{\mu_i}{\lambda_{i}}\|u\|_{\infty}\quad \forall\ u\in C^{b}_{\rm unif}(\R), \ i\in\{1,2\}.
\end{equation}
Furthermore, let
$$
C_{\rm unif}^{2,b}(\R)=\{u\in C_{\rm unif}^b(\R)\,|\, u^{'}(\cdot),\, u^{''}(\cdot)\in C_{\rm unif}^b(\R)\}.
$$

The next Lemma provides a uniform lower/upper bounds and a poitwise lower/upper bounds  for $\chi_2\lambda_{2}V_{2}(\cdot;u)-\chi_1\lambda_1V_{1}(\cdot;u)$ whenever $u\in \mathcal{E}_{\mu}$.

\begin{lem}\label{Mainlem1}
For every $0<\mu<\min\{1, \sqrt{\frac{\lambda_1}{a}},\ \sqrt{\frac{\lambda_2}{a}}\}$ and $u\in \mathcal{E}_{\mu}$, let $V_{i}(\cdot;u)$ be defined as in \eqref{Inverse of u}, then for every $x\in\R$, there holds
\begin{align}\label{Eq_MainLem1-1}
&(\chi_2\lambda_2V_2-\chi_1\lambda_1V_1)(x;u)\leq\min\Big\{  \nonumber\\ &(\chi_2\mu_2\lambda_2-\chi_1\mu_1\lambda_1)_{+}\min\big\{\frac{C_0}{\lambda_2},\frac{\varphi_{\mu}(x)}{\lambda_2-a\mu^2} \big\}+ \chi_1\mu_1\lambda_1\min\big\{ \frac{C_0(\lambda_1-\lambda_2)_+}{\lambda_1\lambda_2},\frac{\varphi_{\mu}(x)(\lambda_1-\lambda_2)_+}{(\lambda_1-a\mu^2)(\lambda_2-a\mu^2)}\big\},\nonumber\\
 &
 (\chi_2\mu_2\lambda_2-\chi_1\mu_1\lambda_1)_{+}\min\big\{\frac{C_0}{\lambda_1},\frac{\varphi_{\mu}(x)}{\lambda_1-a\mu^2} \big\}+ \chi_2\mu_2\lambda_2\min\big\{ \frac{C_0(\lambda_1-\lambda_2)_+}{\lambda_1\lambda_2},\frac{\varphi_{\mu}(x)(\lambda_1-\lambda_2)_+}{(\lambda_1-a\mu^2)(\lambda_2-a\mu^2)}\big\} \Big\}
\end{align}
and
\begin{align}\label{Eq_MainLem1-2}
&(\chi_2\lambda_2V_2-\chi_1\lambda_1V_1)(x;u) \geq\max\Big\{  \nonumber\\
&- (\chi_2\mu_2\lambda_2-\chi_1\mu_1\lambda_1)_{-}\min\big\{\frac{C_0}{\lambda_2},\frac{\varphi_{\mu}(x)}{\lambda_2-a\mu^2} \big\} + \chi_1\mu_1\lambda_1\min\big\{ \frac{C_0(\lambda_1-\lambda_2)_-}{\lambda_1\lambda_2},\frac{\varphi_{\mu}(x)(\lambda_1-\lambda_2)_-}{(\lambda_1-a\mu^2)(\lambda_2-a\mu^2)}\big\},\nonumber\\
&-(\chi_2\mu_2\lambda_2-\chi_1\mu_1\lambda_1)_{-}\min\big\{\frac{C_0}{\lambda_1},\frac{\varphi_{\mu}(x)}{\lambda_1-a\mu^2} \big\} + \chi_2\mu_2\lambda_2\min\big\{ \frac{C_0(\lambda_1-\lambda_2)_-}{\lambda_1\lambda_2},\frac{\varphi_{\mu}(x)(\lambda_1-\lambda_2)_-}{(\lambda_1-a\mu^2)(\lambda_2-a\mu^2)}\big\} \Big\}
\end{align}
where $C_0:=\frac{a}{b+\chi_2\mu_2-\chi_1\mu_1-M}$ and  $M$ is given by \eqref{const-upper-bound}.
 \end{lem}
\begin{proof}
Observe that  for every  $\tau\in \{+,-\}$, $i\in\{1,2\}$, and $x\in\R$, we have
\begin{equation}
\int_{0}^{\infty}\int_{\R}e^{-\lambda_{i}s}\frac{e^{-\frac{|x-z|^2}{4s}}}{\sqrt{4\pi s}}dzds=\frac{1}{\lambda_i}, \quad \int_{0}^{\infty}\int_{\R}(e^{-\lambda_{2}s}-e^{-\lambda_{1}s})_{\tau}\frac{e^{-\frac{|x-z|^2}{4s}}}{\sqrt{4\pi s}}dzds=\frac{(\lambda_{1}-\lambda_2)_{\tau}}{\lambda_1 \lambda_2},
\end{equation}
\begin{equation}
\int_{0}^{\infty}\int_{\R}e^{-\lambda_{i}s}\frac{e^{-\frac{|x-z|^2}{4s}}}{\sqrt{4\pi s}}\varphi_{\mu}(z)dzds=\frac{\varphi_{\mu}(x)}{\lambda_i-a\mu^2},
\end{equation}
and
\begin{equation}
\int_{0}^{\infty}\int_{\R}(e^{-\lambda_{2}s}-e^{-\lambda_{1}s})_{\tau}\frac{e^{-\frac{|x-z|^2}{4s}}}{\sqrt{4\pi s}}\varphi_{\mu}(z)dzds=\frac{(\lambda_{1}-\lambda_2)_{\tau}}{(\lambda_1 -a\mu^2)(\lambda_2-a\mu^2)}\varphi_{\mu}(x).
\end{equation}
Let us set
$
C_{0}=\frac{a}{b+\chi_2\mu_2-\chi_1\mu_1-M}.
$ Hence, since $0\leq u\leq \min\{C_0 , \varphi_{\mu}\}$, we obtain
\begin{align}\label{eq005}
( \chi_2\lambda_2 V_2 -\chi_1\lambda_1V_1)(x;u) &= (\chi_2\mu_2\lambda_2-\chi_1\mu_1\lambda_1)\int_0^{\infty}\int_{\R}e^{-\lambda_2 s}\frac{e^{-\frac{|x-z|^2}{4s}}}{(4\pi s)^{\frac{1}{2}}}u(z)dzds\nonumber \\
 & +\chi_1\mu_1\lambda_1 \int_{0}^{\infty}\int_{\R}(e^{-\lambda_{2}s}-e^{-\lambda_{1}s})\frac{e^{-\frac{|x-z|^2}{4s}}}{(4\pi s)^{\frac{1}{2}}}u(z)dzds \nonumber\\
&\leq (\chi_2\lambda_2\mu_2-\chi_1\lambda_1\mu_1)_{+}\min\Big\{\frac{C_0}{\lambda_2}, \frac{\varphi_{\mu}(x)}{\lambda_2-a\mu^2}\Big\}\nonumber\\
&+\chi_1\mu_1\lambda_1\min\Big\{\frac{C_0(\lambda_1-\lambda_2)_{+}}{\lambda_1\lambda_2},\frac{\varphi_{\mu}(x)(\lambda_1-\lambda_2)_{+}}{(\lambda_1-a\mu^2)(\lambda_{2}-a\mu^2)}\Big\}.
\end{align}
Similarly, we have
\begin{align}\label{eq006}
( \chi_2\lambda_2V_2 -\chi_1\lambda_1V_1)(x;u)
&=\chi_2\mu_2\lambda_2\int_0^{\infty}\int_{\R^N}(e^{-\lambda_2 s}-e^{-\lambda_1 s})\frac{e^{-\frac{|x-z|^2}{4s}}}{(4\pi s)^{\frac{1}{2}}}u(z,t)dzds  \nonumber\\
& +(\chi_2\mu_2\lambda_2-\chi_1\mu_1\lambda_1)\int_{0}^{\infty}\int_{\R^N}e^{-\lambda_{1}s}\frac{e^{-\frac{|x-z|^2}{4s}}}{(4\pi s)^{\frac{1}{2}}}u(z,t)dzds, \nonumber\\
& \leq (\chi_2\lambda_2\mu_2-\chi_1\lambda_1\mu_1)_{+}\min\Big\{\frac{C_0}{\lambda_1}, \frac{\varphi_{\mu}(x)}{\lambda_1-a\mu^2}\Big\}\nonumber\\
&+\chi_2\mu_2\lambda_2\min\Big\{\frac{C_0(\lambda_1-\lambda_2)_{+}}{\lambda_1\lambda_2},\frac{\varphi_{\mu}(x)(\lambda_1-\lambda_2)_{+}}{(\lambda_1-a\mu^2)(\lambda_{2}-a\mu^2)}\Big\}.
\end{align}
On the other hand, we have
\begin{align}\label{eq007}
( \chi_2\lambda_2V_2 -\chi_1\lambda_1V_1)(x;u)
&\geq - (\chi_2\mu_2\lambda_2-\chi_1\mu_1\lambda_1)_{-}\int_0^{\infty}\int_{\R}e^{-\lambda_2 s}\frac{e^{-\frac{|x-z|^2}{4s}}}{(4\pi s)^{\frac{1}{2}}}u(z)dzds\nonumber\\  &-\chi_1\mu_1\lambda_1 \int_{0}^{\infty}\int_{\R}(e^{-\lambda_{2}s}-e^{-\lambda_{1}s})_{-}\frac{e^{-\frac{|x-z|^2}{4s}}}{(4\pi s)^{\frac{1}{2}}}u(z)dzds \nonumber\\
&\geq  -(\chi_2\lambda_2\mu_2-\chi_1\lambda_1\mu_1)_{-}\min\Big\{\frac{C_0}{\lambda_2}, \frac{\varphi_{\mu}(x)}{\lambda_2-a\mu^2}\Big\}\nonumber\\
&-\chi_1\mu_1\lambda_1\min\Big\{\frac{C_0(\lambda_1-\lambda_2)_{-}}{\lambda_1\lambda_2},\frac{\varphi_{\mu}(x)(\lambda_1-\lambda_2)_{-}}
{(\lambda_1-a\mu^2)(\lambda_{2}-a\mu^2)}\Big\}
\end{align}
and
\begin{align}\label{eq008}
( \chi_2\lambda_2V_2 -\chi_1\lambda_1V_1)(x,t;u)
&\geq -\chi_2\mu_2\lambda_2\int_0^{\infty}\int_{\R}(e^{-\lambda_2 s}-e^{-\lambda_1 s})_{-}\frac{e^{-\frac{|x-z|^2}{4s}}}{(4\pi s)^{\frac{1}{2}}}u(z,t)dzds  \nonumber\\
& - (\chi_2\mu_2\lambda_2-\chi_1\mu_1\lambda_1)_{-}\int_{0}^{\infty}\int_{\R}e^{-\lambda_{1}s}\frac{e^{-\frac{|x-z|^2}{4s}}}{(4\pi s)^{\frac{1}{2}}}u(z,t)dzds, \nonumber\\
& \geq -(\chi_2\lambda_2\mu_2-\chi_1\lambda_1\mu_1)_{-}\min\Big\{\frac{C_0}{\lambda_1}, \frac{\varphi_{\mu}(x)}{\lambda_1-a\mu^2}\Big\}\nonumber\\
&-\chi_2\mu_2\lambda_2\min\Big\{\frac{C_0(\lambda_1-\lambda_2)_{-}}{\lambda_1\lambda_2},\frac{\varphi_{\mu}(x)(\lambda_1-\lambda_2)_{-}}{(\lambda_1-a\mu^2)(\lambda_{2}-a\mu^2)}\Big\}.
\end{align}

The Lemma thus follows.
\end{proof}

\begin{rk}\label{remark0}  Let $0<\mu<\min\{1, \sqrt{\frac{\lambda_1}{a}},\ \sqrt{\frac{\lambda_2}{a}}\}$ and $u\in \mathcal{E}_{\mu}$ be given.
\begin{enumerate}
\item[(1)] It follows from Lemma \ref{Mainlem1} that
\begin{equation*}
(\chi_2 \lambda_2V_2-\chi_1 \lambda_1V_1)(x;u)\leq \min\{ M C_{0} , \overline{L}_{\mu}\varphi_{\mu}(x)\}.
\end{equation*}
where $M$ is given by \eqref{const-upper-bound}, $C_0=\frac{a}{b+\chi_2\mu_2-\chi_1\mu_1-M}$, and $\bar{L}_{\mu}$ is given by \eqref{L-upper}.

\item[(2)]It follows from Lemma \ref{Mainlem1} that
\begin{equation*}
(\chi_2 \lambda_2V_2-\chi_1 \lambda_1V_1)(x;u)\geq -\min\{ C_{0}\tilde M, \underline{L}_{\mu}\varphi_{\mu}(x)\}.
\end{equation*}
where $C_0=\frac{a}{b+\chi_2\mu_2-\chi_1\mu_1-M}$, $M$ is given by \eqref{const-upper-bound}, $\tilde M$ is given by \eqref{const-lower-bound},
and $\underline{L}_{\mu}$ is given by
\eqref{L-lower}.
\end{enumerate}
\end{rk}

Next, we present a pointwise and uniform estimate  for $|\partial_{x}(\chi_2V_2-\chi_1V_1)(\cdot;u)|$ whenever $u\in \mathcal{E}_{\mu}.$

\begin{lem}\label{Mainlem2} Let  $0<\mu<\min\{1,\frac{\sqrt{\lambda_1}}{\sqrt{a}},\frac{\sqrt{\lambda_2}}{\sqrt{a}}\}$ be fixed. Let $u\in C^{b}_{\rm unif}(\R)$ and $V_{1}(\cdot;u)\in C^{2,b}_{\rm unif}(\R)$ (resp. $V_{2}(\cdot;u)\in C^{2,b}_{\rm unif}(\R)$)  be the corresponding function satisfying the second equation (resp. third equation) of \eqref{Main-eq1}. Then, for every $i,j\in\{1,2\}$, $x \in\R$, and every $u\in\mathcal{E}_{\mu}$, we have
\begin{align}\label{Eq_Mainlem01}
&| \partial_{x}(\chi_{i}V_{i}-\chi_{j}V_{j})(x;u)|\nonumber \\
&\leq |\chi_i\mu_i-\chi_j\mu_j|\min\Big\{ \frac{C_0}{\sqrt{\lambda_i}}\ ,\ \big(\frac{1}{\sqrt{\lambda_i-a\mu^2}}+\frac{\mu\sqrt{a}}{\lambda_i-a\mu^2}\big)\varphi_{\mu}(x)  \Big\}\nonumber\\
& + \chi_j\mu_j\min\Big\{ \frac{C_0|\sqrt{\lambda_i}-\sqrt{\lambda_j}|}{\sqrt{\lambda_i\lambda_j}}\ ,\ \big(\big|\frac{1}{\sqrt{\lambda_i-a\mu^2}}-\frac{1}{\sqrt{\lambda_j-a\mu^2}}\big|+\frac{\mu\sqrt{a}|\lambda_i-\lambda_j|}{(\lambda_i-a\mu^2)(\lambda_j-a\mu^2)}  \big)\varphi_{\mu}(x) \Big\}.
\end{align}
where $C_0:=\frac{a}{b+\chi_2\mu_2-\chi_1\mu_1-M}$ and $M$ is given by \eqref{const-upper-bound}
\end{lem}

\begin{proof}
For every  $k\in\{1,2\}$, note that
\begin{align*}
 \partial_{x}(\chi_kV_k)(x;u)&=\frac{\chi_k\mu_k}{\sqrt{\pi}}\int_0^\infty\frac{ e^{-\lambda_k s}}{\sqrt{ s}}\Big[\int_{\R}\tau e^{-\tau^2}u(x+2\sqrt{s}\tau )d\tau\Big]ds
\end{align*}
Hence, for every $i,j\in\{1,2\}$, we have
\begin{eqnarray}\label{a-eq1}
 \partial_{x}(\chi_iV_i-\chi_jV_j)(x;u)
&=& \frac{(\chi_i\mu_i-\chi_j\mu_j)}{\sqrt{\pi}}\int_0^\infty\frac{e^{-\lambda_i s}}{\sqrt{s}}\Big[ \int_{\R} \tau e^{-\tau^2}u(x+2\sqrt{s}\tau)d\tau \Big]ds\nonumber\\
& +& \frac{\chi_j\mu_j}{\sqrt{\pi}}\int_0^\infty\frac{(e^{-\lambda_i s}-e^{-\lambda_j s})}{\sqrt{s}}\Big[\int_{\R} \tau e^{-\tau^2}u(x+2\sqrt{s}\tau)d\tau\Big]ds\nonumber\\
\end{eqnarray}
Observe that for every $u\in \mathcal{E}_{\mu}$, $x\in\R$ and $s>0$, we have
\begin{align}\label{a-eq2}
\int_{\R}|\tau|e^{-\tau^2}u(x+2\tau\sqrt{s})d\tau ds& \leq \int_{\R}|\tau|e^{-\tau^2}\varphi_{\mu}(x+2\tau\sqrt{s})d\tau ds\nonumber\\
&=\Big[\int_{\R}|\tau|e^{-(\tau^2+2\tau\mu\sqrt{as})}d\tau \Big]\varphi_{\mu}(x)\nonumber\\
&\leq \Big[\int_{\R}(|\tau|+\mu\sqrt{as})e^{-\tau^2}d\tau \Big]e^{\mu^2as}\varphi_{\mu}(x)\nonumber\\
&=\Big(1+\mu\sqrt{a\pi s}\Big)e^{\mu^2as}\varphi_{\mu}(x)
\end{align} and
\begin{align}\label{a-eq3}
\int_{\R}|\tau|e^{-\tau^2}u(x+2\tau\sqrt{s})d\tau ds& \leq \frac{a}{b+\chi_2\mu_2-\chi_1\mu_1-M}:=C_0.
\end{align}
It follows from \eqref{a-eq1},\eqref{a-eq2} and \eqref{a-eq3} that, for every $u\in\mathcal{E}_\mu,\ x\in\R,\ s>0$ and $i,j\in\{1,2\}$, we have
\begin{align}\label{a-eq4}
&\Big| \partial_{x}(\chi_iV_i-\chi_jV_j)(x;u) \Big|\nonumber\\
& \leq \frac{|\chi_i\mu_i-\chi_j\mu_j|}{\sqrt{\pi}}\min\Big\{ C_0\int_0^\infty\frac{e^{-\lambda_i s}}{\sqrt{s}}ds\ ,\ \varphi_{\mu}(x)\int_0^\infty\frac{e^{-(\lambda_i-a\mu^2)s}}{\sqrt{s}}(1+\mu\sqrt{a\pi s})ds  \Big\}\nonumber\\
&+\frac{\chi_j\mu_j}{\sqrt{\pi}}\min\Big\{C_0\int_0^\infty\frac{|e^{-\lambda_i s}-e^{-\lambda_j s}|}{\sqrt{s}}ds\ ,\ \varphi_{\mu}(x)\int_0^\infty\frac{|e^{-(\lambda_i-a\mu^2) s}-e^{-(\lambda_j -a\mu^2)s}|}{\sqrt{s}}(1+\mu\sqrt{a\pi s})ds\Big\}\nonumber\\
&= |\chi_i\mu_i-\chi_j\mu_j|\min\Big\{ \frac{C_0}{\sqrt{\lambda_i}}\ ,\ \Big(\frac{1}{\sqrt{\lambda_i-a\mu^2}}+\frac{\mu\sqrt{a}}{\lambda_i-a\mu^2}\Big)\varphi_{\mu}(x)  \Big\}\nonumber\\
& + \chi_j\mu_j\min\Big\{ \frac{C_0|\sqrt{\lambda_i}-\sqrt{\lambda_j}|}{\sqrt{\lambda_i\lambda_j}}\ ,\ \Big(\Big|\frac{1}{\sqrt{\lambda_i-a\mu^2}}-\frac{1}{\sqrt{\lambda_j-a\mu^2}}\Big|+\frac{\mu\sqrt{a}|\lambda_i-\lambda_j|}{(\lambda_i-a\mu^2)(\lambda_j-a\mu^2)}  \Big)\varphi_{\mu}(x) \Big\}
\end{align}
The Lemmas thus follows from \eqref{a-eq4}.
\end{proof}

 \begin{rk}\label{remark 1}
 Let $0<\mu<\min\{1, \sqrt{\frac{\lambda_1}{a}},\ \sqrt{\frac{\lambda_2}{a}}\}$ and $u\in \mathcal{E}_{\mu}$ be given.
 \begin{enumerate}
 \item[(1)] It follows from Lemma \ref{Mainlem2} that
 \begin{align}
 |\partial_{x}(\chi_1V_1-\chi_2V_2)(x;u)|\leq K_{\mu}\varphi_{\mu}(x)
 \end{align}
 for every $x\in\R$ and $u\in\mathcal{E}_{\mu}$,  where $K_{\mu}$ is given by
 \eqref{K}.

 \item[(2)] It also follows from Lemma \ref{Mainlem2} that \begin{align}
 |\partial_{x}(\chi_1V_1-\chi_2V_2)(x;u)|\leq \frac{\tilde K a}{b+\chi_2\mu_2-\chi_1\mu_1-M}.
 \end{align}
 for every $x\in\R$ and $u\in\mathcal{E}_{\mu}$, where $\tilde K$ is given by \eqref{D}.
 \end{enumerate}
 \end{rk}

Based on Remarks \ref{remark0} and \ref{remark 1}, we can now present the proof of Theorem \ref{super-sub-solu-thm}.
\begin{proof}[Proof of Theorem \ref{super-sub-solu-thm}]
 For given  $u\in\mathcal{E}_{\mu}$  and $U\in C^{2,1}(\R\times\R)$, let \begin{equation}\label{mathcal L}
\mathcal{L}U=U_{xx}+(c_{\mu}+\partial_{x}(\chi_2V_2-\chi_{1}V_{1})(\cdot;u))U_{x}+(a+(\chi_2\lambda_2V_{2}-\chi_{1}\lambda_1V_{1})(\cdot;u)-(b+\chi_2\mu_2-\chi_1\mu_1)U)U .
\end{equation}
(1)  First, let $C_{0}=\frac{a}{b+\chi_2\mu_2-\chi_1\mu_1-M}$. By Remark \ref{remark0} (1), we have that
\begin{eqnarray}\label{N001}
\mathcal{L}(C_0 )&=& (a+(\chi_2\lambda_2V_{2}-\chi_{1}\lambda_1V_{1})(\cdot;u)-(b+\chi_2\mu_2-\chi_1\mu_1)C_0)C_0\nonumber\\
& \leq & \Big( a +MC_0-(b+\chi_2\mu_2-\chi_{1}\mu_1)C_0\Big)C_0\nonumber\\
&=& 0.
\end{eqnarray}
Hence, we have that $U(x,t)=C_0$ is a super-solution of \eqref{ODE2} on $\R\times\R$.

(2) It follows from Lemmas \ref{Mainlem1} and \ref{Mainlem2}  that
\begin{align*}
\mathcal{L}(\varphi_{\mu})& = \varphi''_{\mu}(x)+(c_{\mu}+\partial_{x}(\chi_2V_2-\chi_1 V_1)(\cdot;u))\varphi_{\mu}'(x)\nonumber\\
&+(a+(\chi_2\lambda_2V_{2}-\chi_1\lambda_1V_1)(\cdot;u)-(b+\chi_2\mu_2-\chi_1\mu_1)\varphi_{\mu})\varphi_{\mu}\nonumber\\
 & = \underbrace{(\varphi''_{\mu}+c_{\mu}\varphi'_{\mu}+a\varphi_{\mu})}_{=0}+\Big(-\mu\sqrt{a}\partial_{x}(\chi_2V_2-\chi_1 V_1)(\cdot;u)+(\chi_2\lambda_2V_{2}-\chi_1\lambda_1V_1)(\cdot;u)\Big)\varphi_{u}\nonumber\\
 &-(b+\chi_2\mu_2-\chi_1\mu_1)\varphi_{\mu}^2\nonumber\\
   & \leq \Big[\mu\sqrt{a}K_{\mu}+\overline{L}_{\mu}- (b+\chi_2\mu_2-\chi_1\mu_1) \Big]\varphi_{\mu}^{2} \leq  0
\end{align*}
whenever \eqref{sup-sub-solu-eq} holds. Hence $U(x,t)=\varphi_{\mu}(x)$ is also a super-solution of \eqref{ODE2} on $\R\times\R$.

(3)  Let $O=(\underline{a}_{\mu,\tilde{\mu},d},\infty)$. Then for $x\in O$, $U_{\mu}^-(x)>0$.
 For $x\in O$, it follows from Lemmas \ref{Mainlem1} and \ref{Mainlem2} that
\begin{align*}
\mathcal{L}(U_{\mu}^{-}
)& = a\mu^2\varphi_{\mu}-a\tilde{\mu}^2d\varphi_{\tilde{\mu}} +(c_{\mu}+\partial_{x}(\chi_2V_2-\chi_1V_1)(\cdot;u))(-\mu\sqrt{a}\varphi_{\mu}+d\sqrt{a}\tilde{\mu}\varphi_{\tilde{\mu}})\nonumber\\
& +(a+( \chi_2\lambda_2V_2-\chi_1\lambda_1V_1)(\cdot;u)-(b+\chi_2\mu_2-\chi_1\mu_1) U_{\mu}^{-})U_{\mu}^{-} \nonumber\\
& =\underbrace{(a\mu^{2}-\sqrt{a}\mu c_{\mu}+a)}_{=0}\varphi_{\mu} +d\underbrace{(\sqrt{a}\tilde{\mu}c_{\mu}-a\tilde{\mu}^{2}-a)}_{=A_{0}}\varphi_{\tilde{\mu}} -\sqrt{a}\mu\partial_x(\chi_2V_2-\chi_1V_1)(\cdot;u)\varphi_{\mu}\nonumber\\
&  + d\sqrt{a}\tilde{\mu}\partial_x(\chi_2V_2-\chi_1\lambda_1V_1)(\cdot;u)\varphi_{\tilde{\mu}}  +(( \chi_2\lambda_2V_2-\chi_1V_1)(\cdot;u)-(b+\chi_2\mu_2-\chi_1\mu_1) U_{\mu}^{-})U_{\mu}^{-}\nonumber\\
& \geq  dA_{0}\varphi_{\tilde{\mu}} -\sqrt{a}K_{\mu}\mu\varphi_{\mu}^2-d\sqrt{a}K_{\mu}\tilde{\mu}\varphi_{\mu}\varphi_{\tilde{\mu}}-\underline{L}_{\mu}\varphi_{\mu}(x)U^{-}_{\mu} -(b+\chi_2\mu_2-\chi_1\mu_1)[U^{-}_{\mu}]^2\nonumber\\
&\geq dA_{0}\varphi_{\tilde{\mu}}-\underbrace{\Big[ \sqrt{a}K_{\mu}\mu+\underline{L}_{\mu}+b+\chi_2\mu_2-\chi_1\mu_1\Big]}_{=A_1}\varphi_{\mu}^2-d^2(b+\chi_2\mu_2-\chi_1\mu_1)\varphi_{\tilde{\mu}}^2 \nonumber\\
& +d\Big[-\sqrt{a}K_{\mu}\tilde{\mu}+\underline{L}_{\mu} +2( b+\chi_2\mu_2-\chi_1\mu_1)\Big]\varphi_{\mu}\varphi_{\tilde{\mu}}\nonumber\\
\end{align*}
Note that  $U_{\mu}^{-}(x)>0$ is equivalent to $\varphi_{\mu}(x)>d\varphi_{\tilde{\mu}}(x)$, which is again equivalent to
$$
d(b+\chi_2\mu_2-\chi_1\mu_1)\varphi_{\mu}(x)\varphi_{\tilde{\mu}}(x)>d^{2}(b+\chi_2\mu_2-\chi_1\mu_1)\varphi^2_{\tilde{\mu}}(x).
$$
Since $A_{1}>0$, thus for $x\in O$, we have
\begin{eqnarray*}
\mathcal{L}U_{\mu}^{-}(x) & \geq &  dA_{0}\varphi_{\tilde{\mu}}(x) -A_{1}\varphi_{\mu}^2(x)\nonumber\\
& & +d\underbrace{\Big[-\sqrt{a}K_{\mu}\tilde{\mu}+\underline{L}_{\mu} + b+\chi_2\mu_2-\chi_1\mu_1\Big]}_{A_{2}}\varphi_{\mu}(x)\varphi_{\tilde{\mu}}(x)\nonumber\\
& =& A_{1}\Big[\frac{dA_{0}}{A_{1}}e^{\sqrt{a}(2\mu-\tilde{\mu})x}-1\Big]\varphi_{\mu}^{2}(x) +dA_{2}\varphi_{\mu}(x)\varphi_{\tilde{\mu}}(x).
\end{eqnarray*}
Note also that, by \eqref{sup-sub-solu-eq},
\begin{eqnarray}\label{Eq1 of Th2}
A_{2}&=& \Big(-\sqrt{a}K_{\mu}\mu -\overline{L}_{\mu}+ b+\chi_2\mu_2-\chi_1\mu_1\Big)+\Big(\overline{L}_{\mu}+\underline{L}_{\mu} -\sqrt{a}K_{\mu}(\tilde{\mu}-\mu) \Big)\nonumber\\
&\geq &  \overline{L}_{\mu}+\underline{L}_{\mu} -\sqrt{a}K_{\mu}(\tilde{\mu}-\mu)\nonumber\\
&\geq & 0,
\end{eqnarray}
whenever $\sqrt{a}K_{\mu}(\tilde{\mu}-\mu)\leq\overline{L}_{\mu}+\underline{L}_{\mu}$.
Observe that
$$
 A_{0}=\frac{a(\tilde{\mu}-\mu)(1-\mu\tilde{\mu})}{\mu}>0,\quad \forall\ 0<\mu<\tilde{\mu}<1.
$$
 Furthermore, we have that $U_{\mu}^{-}(x)>0$ implies that $x>0$ for $d>1$. Thus, for every $ d\geq d_{0}:= \max\{1, \frac{A_{1}}{A_{0}}\}$, we have that
\begin{equation}\label{E1}
\mathcal{L}U_{\mu}^{-}(x) > 0
\end{equation}
whenever $x\in O$, $\sqrt{a}K_{\mu}(\tilde{\mu}-\mu)\leq \overline{L}_{\mu}+\underline{L}_{\mu}$, and $\mu<\tilde{\mu}< \min\{1, 2\mu, \sqrt{\frac{\lambda_1}{a}},\ \sqrt{\frac{\lambda_2}{a}}\}$. Hence $U(x,t)=U_{\mu}^-(x)$ is a sub-solution of \eqref{ODE2} on $(\underline{a}_{\mu,\tilde{\mu},d},\infty)\times\R$.

\smallskip

(4) Observe that $$ a -\frac{ a \tilde M}{b+\chi_2\mu_2-\chi_1\mu_1-M}=\frac{a(b+\chi_2\mu_2-\chi_1\mu_1-M- \tilde M)}{b+\chi_2\mu_2-\chi_1\mu_1-M}>0.$$

Hence, for $0<\delta\ll 1$, we have that
\begin{align*}
\mathcal{L}(U_{\mu}^{-}(x_{\delta}))
&=(a+(\chi_2\lambda_2V_{2}-\chi_1\lambda_1V_1)(x;u)-(b+\chi_2\mu_2-\chi_1\mu_1)U_{\mu}^-(x_\delta))U_{\mu}^-(x_\delta)\nonumber\\
& \geq  (a -\frac{ a\tilde M}{b+\chi_2\mu_2-\chi_1\mu_1-M}-(b+\chi_2\mu_2-\chi_1\mu_1)U_{\mu}^{-}(x_{\delta}))U_{\mu}^{-}(x_{\delta})\\
\end{align*}
where $x_\delta=\underline{a}_{\mu,\tilde{\mu},d}+\delta$. This implies that $U(x,t)=U_{\mu}^-(x_\delta)$ is
a sub-solution of \eqref{ODE2} on $\R\times\R$.
\end{proof}

\section{Traveling wave solutions}

In this section we study the existence and nonexistence  of traveling wave solutions of \eqref{Main-eq1} connecting $(\frac{a}{b},\frac{a\mu_1}{b\lambda_1},\frac{a\mu_2}{b\lambda_2})$ and
$(0,0,0)$, and prove Theorems A and B.

\subsection{Proof of Theorem  A}

In this subsection, we  prove Theorem A. To this end,  we first prove  the following important result.

\begin{tm}\label{existence-tv-thm}
 Assume (H).  Then \eqref{Main-eq1} has a traveling wave solution
$(u(x,t),v_1(x,t),v_2(x,t))=(U(x-c_\mu t),V_1(x-c_\mu t),V_2(x-c_{\mu}t))$ satisfying
$$
\lim_{x\to-\infty}U(x)=\frac{a}{b} \quad \text{and}\quad
\lim_{x\to\infty}\frac{U(x)}{e^{-\sqrt a \mu x}}=1
$$
where $c_{\mu}=\sqrt{a}(\mu+\frac{1}{\mu})$.
\end{tm}


In order to prove Theorem \ref{existence-tv-thm}, we first prove some lemmas. These Lemmas extend some of the results established in \cite{SaSh2}, so some details might be omitted in their proofs. The reader is referred to the proofs of Lemmas 3.2, 3.3, 3.5 and 3.6 in \cite{SaSh2} for more details.

 In the remaining part of this subsection we shall suppose that \eqref{sup-sub-solu-eq}  holds and $\tilde \mu$ is fixed, where $\tilde \mu$ satisfies
   $$\mu<\tilde{\mu}<\min\{1,\sqrt{\frac{\lambda_1}{a}},\sqrt{\frac{\lambda_1}{a}},2\mu\} \quad \text{and} \quad \sqrt{a}K_\mu(\tilde{u}-\mu)<\overline{L}_\mu+\underline{L}_\mu.
    $$
 Furthermore, we choose  $d=d_0(\chi_1,\mu_1,\lambda_1,\chi_2,\mu_2,\lambda_2,\mu)$ to be the constant given by Theorem \ref{super-sub-solu-thm} and to be fixed. Fix $u\in\mathcal{E}_{\mu}$. For given $u_0\in C_{\rm unif}^b(\R)$, let
  $U(x,t;u_0,u)$ be the solution of \eqref{ODE2} with
$U(x,0;u_0,u)=u_0(x)$. By the arguments in the proofs of Theorem 1.1 and Theorem 1.5 in \cite{SaSh1}, we have $U(x,t; U_{\mu}^+,u)$ exists for all $t>0$ and
${ U(\cdot,\cdot;{ U_{\mu}^+},u)}\in C([0,\infty),C^{b}_{\rm unif}(\R))\cap C^{1}((0\ ,\ \infty),C^{b}_{\rm unif}(\R))\cap C^{2,1}(\R\times(0,\ \infty))$ satisfying
\begin{equation}
U(\cdot,\cdot; U_{\mu}^+,u), U_{x}(\cdot,\cdot; U_{\mu}^+,u),U_{xx}(\cdot,t; U_{\mu}^+,u),U_{t}(\cdot,\cdot; U_{\mu}^+,u)\in  C^{\theta}((0, \infty),C_{\rm unif}^{\nu}(\R))
\end{equation}
for $0<\theta, \nu \ll 1$.

\begin{lem} \label{lm1} Assume (H).
 Then for every $u\in \mathcal{E}_{\mu}$, the following hold.
\begin{description}
\item[(i)] $0\leq U(\cdot,t; U_{\mu}^+,u)\leq U_{ \mu}^{+}(\cdot)$ for every $t\geq 0.$
\item[(ii)] $U(\cdot,t_{2}; U_{\mu}^+,u)\leq U(\cdot,t_{1}; U_{\mu}^+, u) $ for every $0\leq t_{1}\leq t_{2}$.
\end{description}
\end{lem}
\begin{proof}
(i)   Note that $0\leq U^{+}_{\mu}(\cdot)\leq \frac{a}{b+\chi_2\mu_2-\chi_1\mu_1-M}$. Then by
comparison principle for parabolic equations and Theorem \ref{super-sub-solu-thm}(1), we have
\begin{equation*}
0\leq U(x,t; U_{\mu}^+,u)\leq \frac{a}{b+\chi_2\mu_2-\chi_1\mu_1-M} \quad \forall\ x\in\R,\ t\geq 0.
\end{equation*}

Similarly, note that $0\leq U_{\mu}^+(x)\le\varphi_{\mu}(x)$.  Then by  comparison principle for parabolic equations and
Theorem \ref{super-sub-solu-thm}(2)  again, we have
\begin{equation*}
U(x,t;U_{\mu}^+,u)\leq \varphi_{\mu}(x) \ \quad \forall\ x\in\R,\ t\geq 0.
\end{equation*}
Thus $U(\cdot,t;U_{\mu}^+,u)\leq U^{+}_{\mu}$. This complete of (i).

(ii)  For $0\leq t_{1}\leq t_{2}$, since
$$
U(\cdot,t_{2};U_{\mu}^+,u)=U(\cdot,t_{1} ;U(\cdot,t_{2}-t_{1};U_{\mu}^+,u),u)
$$
and by (i), $U(\cdot,t_{2}-t_{1};U_{\mu}^+,u)\leq U^{+}_{\mu} $, (ii) follows from comparison principle for parabolic equations.
\end{proof}

Let us define $U(x; u)$ to be
\begin{equation}
\label{U-eq}
{
U(x; u)=\lim_{t\to\infty}U(x,t; U^{+}_{\mu}, u)=\inf_{t>0}U(x,t; U^{+}_{\mu}, u).
}
\end{equation}
 By the a priori estimates for parabolic equations, the limit in \eqref{U-eq} is uniform in $x$ in compact subsets of $\R$
and $U(\cdot;u)\in C_{\rm unif}^b(\R)$.
Next we prove that the function $u\in \mathcal{E}_{\mu}\to U(\cdot;u)\in\mathcal{E}_{\mu}$.

\begin{lem}\label{lm2}
Assume (H).  Then,
\begin{equation}
U(x;u)\geq\begin{cases}  U^{-}_{\mu}(x),\quad x\ge \underline{a}_{\mu,\tilde{\mu},d}\cr
U_{\mu}^-(x_\delta),\quad x\le x_\delta=\underline{a}_{\mu,\tilde{\mu},d}+\delta
\end{cases}
\end{equation}\label{Eq2 of Th2}
for every $u\in\mathcal{E}_{\mu}$,  $t\geq 0$, and $0<\delta\ll 1$.
\end{lem}

\begin{proof} Let $u\in \mathcal{E}_{\mu}$ be fixed.  Let $O=(\underline{a}_{\mu,\tilde{\mu},d},\infty)$.
Note that $U_{\mu}^-(\underline{a}_{\mu,\tilde{\mu},d})=0$. By Theorem \ref{super-sub-solu-thm}(3),
$U_{\mu}^-(x)$ is a sub-solution of \eqref{ODE2} on $O\times (0,\infty)$.
Note also that $U_{\mu}^+(x)\ge U_{\mu}^-(x)$ for $x\ge \underline{a}_{\mu,\tilde{\mu},d}$ and $U(\underline{a}_{\mu,\tilde{\mu},d},t;U_{\mu}^+, u)>0$
for all $t\ge 0$. Then by comparison principle for parabolic equations, we have that
$$
U(x,t;U_{\mu}^+,u)\ge U_{\mu}^-(x)\quad \forall \,\, x\ge \underline{a}_{\mu,\tilde{\mu},d},\,\, t\ge 0.
$$

Now for any $0<\delta\ll 1$, by Theorem \ref{super-sub-solu-thm}(4), $U(x,t)=U_{\mu}^-(x_\delta)$ is a sub-solution of
\eqref{ODE2} on $\R\times \R$. Note that $U_{\mu}^+(x)\ge U_\mu^-(x_\delta)$ for $x\le x_\delta$ and
$U(x_\delta,t;U_{\mu}^+,u)\ge U_{\mu}^-(x_\delta)$ for $t\ge 0$. Then by comparison principle for parabolic equations again,
$$
U(x,t;U_{\mu}^+,u)\ge U_{\mu}^-(x_\delta)\quad \forall\,\, x\le x_\delta,\, \, t>0.
$$
The lemma then follows.
\end{proof}

\begin{rk}\label{Remark-lower-bound-for -solution}
 It follows from Lemmas \ref{lm1} and \ref{lm2} that if \eqref{sup-sub-solu-eq} holds, then
$$
U_{\mu,\delta}^{-}(\cdot)\leq U(\cdot,t;U_{\mu}^+,u)\leq U^{+}_{\mu}(\cdot)$$
for every $u\in\mathcal{E}_{\mu}$, $t\geq0$ and $0\le \delta\ll 1$, where
$$
U_{\mu,\delta}^-(x)=\begin{cases}  U^{-}_{\mu}(x),\quad x\ge \underline{a}_{\mu,\tilde{\mu},d}+\delta\cr
U_\mu^-(x_\delta),\quad x\le x_\delta=\underline{a}_{\mu,\tilde{\mu},d}+\delta.
\end{cases}
$$
 This implies that $$
U_{\mu,\delta}^{-}(\cdot)\leq U(\cdot;u)\leq U^{+}_{\mu}(\cdot)$$
for every $u\in\mathcal{E}_{\mu}$. Hence  $u\in\mathcal{E}_{\mu}\mapsto U(\cdot;u)\in \mathcal{E}_{\mu}.$
\end{rk}

\begin{lem}\label{MainLem02}
Assume (H).
 Then for every $u\in \mathcal{E}_{\mu}$ the associated function $U(\cdot;u)$ satisfied the elliptic equation,
\begin{equation}\label{Eq_MainLem02}
0=U_{xx}+(c_{\mu}+\partial_{x}(\chi_2V_2-\chi_1 V_1)(\cdot;u))U_{x}+(a+(\chi_2\lambda_2V_2-\chi_1\lambda_1 V_1)(\cdot;u)-(b+\chi_2\mu_2-\chi_1\mu_1)U)U
\end{equation}
\end{lem}

\begin{proof} The following arguments generalized the arguments used in the proof of Lemma 4.6 in \cite{SaSh2}. Hence we refer to \cite{SaSh2} for the proofs of the estimates stated below.

 Let $\{t_{n}\}_{n\geq 1}$ be an increasing sequence of positive real numbers converging to $\infty$. For every $n\geq 1$, define $U_{n}(x,t)=U(x,t+t_{n}; U_{\mu}^+, u)$ for every $x\in\R, \ t\geq 0$.
For every $n$, $U_{n}$ solves the PDE
\begin{equation*}
\begin{cases}
\partial_{t}U_{n}=\partial_{xx}U_{n}+(c_{\mu}+\partial_x(\chi_2V_2-\chi_1V_1)(\cdot;u))\partial_{x}U_{n}\\
\qquad \qquad +(a+(\chi_2\lambda_2V_2-\chi_1\lambda_1)(\cdot;u)-(b+\chi_2\mu_2-\chi_1\mu_1)U_{n})U_{n}\\
U_{n}(\cdot,0)=U(\cdot,t_{n}; U_{\mu}^+, u).
\end{cases}
\end{equation*}

Let $\{T(t)\}_{t\geq 0}$ be the analytic semigroup on $C^{b}_{\rm unif}(\R)$ generated by $\Delta-I$
 and  let $X^{\beta}={\rm Dom}((I-\Delta)^{\beta})$ be the fractional power spaces of $I-\Delta$ on $C_{\rm unif}^b(\R)$ ($\beta\in [0,1]$).

 The variation of constant formula and the fact that $\partial_{xx}V_{i}(\cdot;u)-\lambda_{i}V=-\mu_i u$ yield that
\begin{eqnarray}\label{variation -of-const}
U(\cdot,t;U_{\mu}^+, u)
&=& \underbrace{T(t)U_{\mu}^{+}}_{I_{1}(t)}+ \underbrace{\int_{0}^{t}T(t-s)(((c_{\mu}+\partial_{x}(\chi_2V_2-\chi_1V_1)(\cdot;u))U(\cdot,s; U_{\mu}^+, u))_{x})(s)ds}_{I_{2}(t)}\nonumber \\
& +&\underbrace{\int_{0}^{t}T(t-s)(1+a+(\chi_2\lambda_2V_2-\chi_1\lambda_1V_1)(\cdot; u))U(\cdot,s;U_{\mu}^+, u)ds}_{I_{3}(t)}\nonumber\\
&-&(b+\chi_2\mu_2-\chi_1\mu_1)\underbrace{\int_{0}^{t}T(t-s)U^{2}(\cdot,s;U_{\mu}^+, u)ds}_{I_{4}(t)}.\nonumber\\
\end{eqnarray}
Let $0<\beta<\frac{1}{2}$ be fixed. There is a positive constant $C_{\beta}$,  (see \cite{Dan Henry}), such that
\begin{equation*}
\|I_{1}(t)\|_{X^{\beta}}\leq \frac{aC_\beta t^{-\beta}e^{-t}}{b+\chi_2\mu_2-\chi_1\mu_1-M},
\end{equation*}
\begin{equation*}\|I_{2}(t)\|_{X^{\beta}}\leq  \frac{aC_{\beta}}{b+\chi_2\mu_2-\chi_1\mu_1-M}(c_{\mu}+\frac{a\tilde K}{b+\chi_2\mu_2-\chi_1\mu_1-M})\Gamma(\frac{1}{2}-\beta),
\end{equation*}
\smallskip
$$
\|I_{3}(t)\|_{X^{\beta}}
 \leq  \frac{aC_{\beta}((1+a)(b+\chi_2\mu_2-\chi_1\mu_1-M) +a(\chi_2\mu_2+\chi_1\mu_1))}{(b+\chi_2\mu_2-\chi_1\mu_1-M)^2}\Gamma(1-\beta),
 $$
 and
 $$ \ \ \|I_{4}(t)\|_{X^{\beta}}\leq \frac{a^2 C_{\beta}}{(b+\chi_2\mu_2-\chi_1\mu_1-M)^{2}}\Gamma(1-\beta).
$$
Note that we have used Lemma \ref{Mainlem2}, mainly  the fact that $|\partial_{x}(\chi_2V_{2}-\chi_1V_1)(\cdot;u)|\leq \frac{ a\tilde K}{b+\chi_2\mu_2-\chi_1\mu_1-M}$, to obtain the uniform upper bound estimates for $\|I_{2}(t)\|_{X^{\beta}}$. Therefore, for every $T>0$ we have that
\begin{equation}\label{Eq_Convergence01}
\sup_{t\geq T}\|U(\cdot,t;U_{\mu}^+,u)\|_{X^{\beta}}\leq M_{T}<\infty,
\end{equation}
where
\begin{equation}\label{Eq_Conv02}
M_{T}= \frac{aC_{\beta}(1+ \Gamma(1-\beta)+\Gamma(\frac{1}{2}-\beta))}{b+\chi_2\mu_2-\chi_1\mu_1-M}\Big[\frac{T^{-\beta}}{e^{T}}+ \big( c_{\mu}+  \frac{(1+a)(b+\chi_2\mu_2)+a(\tilde K+(\chi_2\mu_2+\chi_1\mu_1))}{b+\chi_2\mu_2-\chi_1\mu_1-M}\big)\Big].
\end{equation}
Hence, it follows  from \eqref{Eq_Convergence01} that
\begin{equation}\label{Eqq000}
\sup_{n\geq 1, t\geq 0}\|U_{n}(\cdot,t)\|_{X^{\beta}}\leq M_{t_{1}}<\infty.
\end{equation}
Next, for every $t,h\geq 0$ and $n\geq 1$, we have that
\begin{equation}\label{Eqq00}
\|I_{1}(t+h+t_{n})-I_{1}(t+t_{n})\|_{X^{\beta}}\leq C_{\beta}h^{\beta}(t+t_{n})^{-\beta}e^{-(t+t_n)}\|U_{\mu}^{+}\|_{\infty}\leq C_{\beta}h^{\beta}t_{1}^{-\beta}e^{-t_1}\|U_{\mu}^{+}\|_{\infty},
\end{equation}
\begin{align}\label{Eqq02}
&\|I_{2}(t+t_n+h)-I_{2}(t+t_n)\|_{X^{\beta}}\nonumber\\
&\leq \frac{aC_{\beta}}{b+\chi_2\mu_2-\chi_1\mu_1-M}(c_\mu+\frac{ a\tilde{K}}{b+\chi_2\mu_2-\chi_1\mu_1-M})\Big[h^{\beta}\Gamma(\frac{1}{2	}-\beta)+\frac{h^{\frac{1}{2}-\beta}}{\frac{1}{2}-\beta} \Big] ,
\end{align}
\begin{align}\label{Eqq01}
&\|I_{3}(t+h+t_n)-I_{3}(t+t_n)\|_{X^{\beta}}\nonumber\\
&\leq \frac{aC_{\beta}((1+a)(b+\chi_2\mu_2-\chi_1\mu_1-M)+a(\chi_2\mu_2+\chi_1\mu_1))}
{(b+\chi_2\mu_2-\chi_1\mu_1-M)^2}\Big[h^{\beta}\Gamma(1-\beta)+\frac{h^{1-\beta}}{1-\beta} \Big] ,
\end{align}
and
\begin{eqnarray}\label{Eqq03}
\|I_{4}(t+t_n+h)-I_{4}(t+t_n)\|_{X^{\beta}}\leq\frac{a^2 C_{\beta}}{(b+\chi_2\mu_2-\chi_1\mu_1-M)^2}\Big[h^{\beta}\Gamma(1-\beta)+\frac{h^{1-\beta}}{1-\beta} \Big].
\end{eqnarray}
It follows from inequalities \eqref{Eqq000}, \eqref{Eqq00}, \eqref{Eqq01}, \eqref{Eqq02} and \eqref{Eqq03}, the functions $U_{n} : [0, \infty)\to X^{\beta}$ are uniformly bounded and equicontinuous.
Since $X^{\beta}$ is continuously imbedded in $C^{\nu}(\R)$ for every $0\leq \nu<2\beta$ (See \cite{Dan Henry}),
therefore, the Arzela-Ascoli Theorem and Theorem 3.15 in   \cite{Friedman}, imply that there is a function $\tilde{U}(\cdot,\cdot;u)\in C^{2,1}(\R\times(0,\infty))$ and a subsequence $\{U_{n'}\}_{n\geq 1}$ of $\{U_{n}\}_{n\geq 1}$ such that $U_{n'}\to \tilde{U}$ in $C^{2,1}_{loc}(\R\times(0, \infty))$ as $n\to \infty$ and $\tilde{U}(\cdot,\cdot;u)$ solves the PDE
$$
\begin{cases}
\partial_{t}\tilde{U}=\partial_{xx}\tilde{U}+(c_{\mu}+\partial_{x}(\chi_2 V_2-\chi_1V_1)(\cdot;u)\partial_{x}\tilde{U}\\
\quad \qquad+(a+(\chi_2\lambda_2V_2-\chi_1\lambda_1V_1)(\cdot;u)-(b+\chi_2\mu_2-\chi_1\mu_1)\tilde{U})\tilde{U}, \ t>0\\
\tilde{U}(x,0)=\lim_{n\to \infty}U(x,t_{n'}; U_{\mu}^+, u).
\end{cases}
$$
But $U(x;u)=\lim_{t\to \infty}U(x,t; U_{\mu}^+, u)$ and $t_{n'}\to \infty$ as $n\to \infty$, hence $\tilde{U}(x,t;u)=U(x;u)$ for every $x\in \R,\ t\geq 0$. Hence $U(\cdot;u)$ solves \eqref{Eq_MainLem02}.
\end{proof}

\begin{lem}
\label{aux-lm}  Assume (H).
  Then, for any given $u\in\mathcal{E}_{\mu}$,
\eqref{Eq_MainLem02} has a unique bounded non-negative solution satisfying that
\begin{equation}
\label{aux-eq1}
\liminf_{x\to -\infty}U(x)>0\quad {\rm and}\quad \lim_{x\to\infty}\frac{U(x)}{e^{- \sqrt a\mu x}}=1.
\end{equation}
\end{lem}
The proof of Lemma \ref{aux-lm} follows from  \cite[Lemma 3.6]{SaSh2}.

We now prove Theorem \ref{existence-tv-thm}.

\begin{proof}[Proof of Theorem  \ref{existence-tv-thm}]
Following the proof of Theorem 3.1 in \cite{SaSh2},
 let us consider the normed linear space  $\mathcal{E}=C^{b}_{\rm unif}(\R)$ endowed with the norm
$$\|u\|_{\ast}=\sum_{n=1}^{\infty}\frac{1}{2^n}\|u\|_{L^{\infty}([-n,\ n])}. $$
For every $u\in\mathcal{E}_{\mu}$ we have that
$$\|u\|_{\ast}\leq \frac{a}{b+\chi_2\mu_2-\chi_1\mu_1-M}. $$
Hence $\mathcal{E}_{\mu}$ is a bounded convex subset of $\mathcal{E}$. Furthermore, since the convergence in $\mathcal{E}$ implies the pointwise convergence, then $\mathcal{E}_{\mu}$ is a closed, bounded, and convex subset of $\mathcal{E}$. Furthermore, a sequence of functions in $\mathcal{E}_{\mu}$ converges with respect to norm $\|\cdot\|_{\ast}$ if and only if it  converges locally uniformly on $\R$.

We prove that the mapping $\mathcal{E}_{\mu}\ni u\mapsto U(\cdot;u)\in\mathcal{E}_{\mu}$ has a fixed point. We divide the proof in three steps.

\smallskip

\noindent {\bf Step 1.} In this step, we prove that the mapping $\mathcal{E}_{\mu}\ni u\mapsto U(\cdot;u)\in \mathcal{E}_{\mu}$ is compact.

 Let $\{u_{n}\}_{n\geq 1}$ be a sequence of elements of $\mathcal{E}_{\mu}$. Since $U(\cdot;u_{n})\in \mathcal{E}_{\mu}$ for every $n\geq 1$ then $\{U(\cdot;u_{n})\}_{n\geq 1}$ is clearly uniformly bounded by $\frac{a}{b+\chi_2\mu_2-\chi_1\mu_1-M}$. Using inequality \eqref{Eq_Convergence01}, we have that
\begin{equation*}
\sup_{t\geq 1}\|U(\cdot,t; U_{\mu}^+,u_{n})\|_{X^{\beta}}\leq M_{1}
\end{equation*}
for all $n\geq 1$ where $M_{1}$ is given by \eqref{Eq_Conv02}. Therefore there is $0<\nu\ll 1$ such that
\begin{equation}\label{Proof-MainTh3- Eq1}
\sup_{t\geq 1}\|U(\cdot,t; U_{\mu}^+,u_{n})\|_{C^{\nu}_{\rm unif}(\R)}\leq \tilde{M_{1}}
\end{equation} for every $n\geq 1$ where $\tilde{M_{1}}$ is a constant depending only on $M_{1}$. Since for every $n\geq 1$ and every $x\in\R$, we have that $U(x,t; U_{\mu}^+,u_{n})\to U(x;u_{n})$ as $t\to \infty,$ then it follows from \eqref{Proof-MainTh3- Eq1} that
\begin{equation}\label{Prof-MainTh3- Eq2}
\|U(\cdot;u_{n})\|_{C^{\nu}_{\rm unif}}\leq \tilde{M_{1}}
\end{equation} for every $n\geq 1$. Which implies that the sequence $\{U(\cdot;u_{n})\}_{n\geq 1}$ is equicontinuous. The Arzela-Ascoli's Theorem implies that there is a subsequence $\{U(\cdot;u_{n'})\}_{n\geq 1}$ of the sequence $\{U(\cdot;u_{n})\}_{n\geq 1}$ and a function $U\in C(\R)$ such that $\{U(\cdot;u_{n'})\}_{n\geq 1}$ converges to $U$ locally uniformly on $\R$. Furthermore, the function $U$ satisfies inequality \eqref{Prof-MainTh3- Eq2}. Combining this with the fact  $U_{\mu}^{-}(x)\leq U(x;u_{n'})\leq U_{\mu}^{+}(x)$ for every $x\in\R$ and $n\geq 1$, by letting $n$ goes to infinity, we obtain that  $U\in \mathcal{E}_{\mu}$. Hence  the mapping $\mathcal{E}_{\mu}\ni u\mapsto U(\cdot;u)\in \mathcal{E}_{\mu}$ is compact.

\smallskip

\noindent{\bf Step 2.} In this step, we prove that the mapping $\mathcal{E}_{\mu}\ni u\mapsto U(\cdot;u)\in \mathcal{E}_{\mu}$ is continuous.
 This follows from the arguments used in the proof of Step 2, Theorem 3.1, \cite{SaSh2}

Now by Schauder's Fixed Point Theorem, there is $U\in\mathcal{E}_{\mu}$ such that $U(\cdot;U)=U(\cdot)$. Then
$(U(x),V(x;U))$ is a stationary solution of \eqref{Main-eq2} with $c=c_\mu$. It is clear that
$$
\lim_{x\to\infty}\frac{U(x)}{e^{- \sqrt a\mu x}}=1.
$$

\noindent {\bf Step 3.} We claim that
$$
\lim_{x\to -\infty}U(x)=\frac{a}{b}.
$$
For otherwise, we may assume that there is $x_n\to -\infty$ such that $U(x_n)\to \lambda\not =\frac{a}{b}$ as $n\to\infty$. Define $U_{n}(x)=U(x+x_{n})$ for every $x\in\R$ and $n\geq 1$. By the arguments of Lemma \ref{MainLem02},  there is a subsequence $\{U_{n'}\}_{n\geq 1}$ of $\{U_{n}\}_{n\geq 1}$ and a function  $U^*\in C_{\rm unif}^b(\R)$ such that $\|U_{n'}-U^*\|_{\ast}\to 0$ as $n\to \infty$. Moreover, $(U^*,V_1(\cdot;U^*),V_2(\cdot;U^*))$ is also a stationary solution of \eqref{Main-eq2} with $c=c_\mu$.

\smallskip

\noindent {\bf Claim 1.} $\inf_{x\in\R}U^{*}(x)>0$.

 Indeed, let $0< \delta\ll 1$ be fixed. For  every $x\in\R$, there $N_{x}\gg 1$ such that $x+x_{n'}< x_{\delta}$ for all $n\geq N_{x}$. Hence, It follows from Remark \ref{Remark-lower-bound-for -solution} that  $$0< U_{\mu,\delta}^{-}(x_{\delta})\leq U(x+x_{n'}) \ \forall\ n\geq N_{x}.$$
Letting $n$ goes to infinity in the last inequality, we obtain that $U_{\mu,\delta}^{-}(x_{\delta})\leq U^{*}(x)$ for every $x\in\R$. The claim thus follows.

\noindent {\bf Claim 2.}  $U^*(\cdot)=\frac{a}{b}$.

  Note that $K={ M + \tilde M}$. Note also  that the function $(\tilde U(x,t),\tilde V_1(x,t),\tilde V_2(x,t))=(U^{*}(x-c_{\mu}t),V_1(x-c_{\mu}t,U^{*}),V_2(x-c_{\mu}t,U^{*}))$ solves \eqref{Main-eq1}. Then by \cite[Theorem B]{SaSh4} and Claim 1,
 $$
 \lim_{t\to\infty}\sup_{x\in\R}|U^*(x-c_\mu t)-\frac{a}{b}|=0.
 $$
 This implies that $U^*(x)=\frac{a}{b}$ for any $x\in\R$ and the claim thus follows.

 By Claim 2, we have $\lim_{n\to \infty} U(x_n)=U^*(0)=\frac{a}{b}$, which contracts to $\lim_{n\to \infty} U(x_n)=U^*(0)=\lambda\not =\frac{a}{b}$.
\end{proof}

 Now, we are ready to prove Theorem A.

\begin{proof}[Proof of Theorem A] Using \eqref{mu-lim}, we have
  \begin{equation}\label{zzz0}
  \lim_{\mu\to 0^+}\Big[\mu\sqrt{a}K_{\mu}+\overline{L}_{\mu}\Big]=M.
  \end{equation}
  This combined with the fact that $\tilde M\geq0$ and $M+\tilde M+\chi_1\mu_1<b+\chi_2\mu_2$ imply that there is $0<\mu_0\ll 1$ such that
  \begin{equation}\label{zzz1}
  \mu\sqrt{a}K_{\mu}+\overline{L}_{\mu}\leq b+\chi_2\mu_2-\chi_1\mu_1, \quad \forall\ 0<\mu<\mu_0.
  \end{equation}
Next, let us define $\mu^{*}$ to be
\begin{align*}
\mu^{*}:=\sup\Big\{ \bar{\mu}\in \Big(0\ ,\ \min\{ 1,\sqrt{\frac{\lambda_1}{a}},\sqrt{\frac{\lambda_2}{a}} \}\Big)   \ : \ \mu\sqrt{a}K_{\mu}+\overline{L}_{\mu}\leq b+\chi_2\mu_2-\chi_1\mu_1,\ \forall\ 0<\mu\leq \bar{\mu}\Big\}
\end{align*}
and
\begin{equation}
c^{*}(\chi_1,\mu_1,\lambda_1,\chi_2,\mu_2,\lambda_2):=\lim_{\mu\to \mu^{*-}}c_{\mu}
\end{equation}
where $c_{\mu}=\sqrt{a}(\mu+\frac{1}{\mu})$. Clearly, it follows from \eqref{zzz1} that $\mu^*\geq \mu_0>0$.  Let $c>c^{*}(\chi_1,\mu_1,\lambda_1$, $\chi_2,\mu_2,\lambda_2)$ be given and let $\mu\in(0, \mu^*)$ be the unique solution of the equation $c_{\mu}=c$. Then $\mu, K_{\mu}$ and $\overline{L}_\mu$ satisfy \eqref{sup-sub-solu-eq}. It follows from Theorem \ref{existence-tv-thm}, that  \eqref{Main-eq01} has a traveling wave solution $(U(x,t),V_{1}(x,t),V_2(x,t))=(U(x-ct),V_1(x-ct),V_2(x-ct))$ with speed c connecting $(\frac{a}{b},\frac{a\mu_1}{b\lambda_1},\frac{a\mu_2}{b\lambda_2})$ and $(0,0,0)$. Moreover $\lim_{z\to\infty}\frac{U(z)}{e^{-\mu z}}=1$.

Note that
$$
\lim_{(\chi_{1},\chi_2)\to (0^+,0^+)}K_{\mu}=\lim_{(\chi_{1},\chi_2)\to (0^+,0^+)}\overline{L}_{\mu}=0.
$$ Thus, we have that
$$\lim_{(\chi_{1},\chi_2)\to (0^+,0^+)}\mu^{*}(\chi_1\mu_1,\lambda_1,\chi_2,\mu_2,\lambda_2)=\min\{1,\sqrt{\frac{\lambda_1}{a}}, \sqrt{\frac{\lambda_2}{a}}\}.$$
From what it follows that
\begin{equation*}
\lim_{(\chi_{1},\chi_2)\to (0^+,0^+)}c^{*}(\chi_1,\mu_1,\lambda_1,\chi_2,\mu_2,\lambda_2)=\begin{cases}
2\sqrt{a}\quad \text{if}\quad  a\leq \min\{\lambda_1, \lambda_2\}\\
\frac{a+\lambda_1}{\sqrt{\lambda_1}}\quad \text{if}\quad  \lambda_1\leq \min\{a, \lambda_2\}\\
\frac{a+\lambda_2}{\sqrt{\lambda_2}}\quad \text{if}\quad  \lambda_2\leq \min\{a, \lambda_1\}
\end{cases},\quad \forall \lambda_1,\lambda_2,\mu_1,\mu_2>0.
\end{equation*}
\end{proof}

\medskip

Next, for completeness,  we present the proof of Remark B.
\begin{proof}[Proof of Remark B]
Observe that if $\lambda_{1}=\lambda_2=\lambda$ then $$
M=(\chi_2\mu_2-\chi_1\mu_1)_+,\quad  \tilde M=(\chi_2\mu_2-\chi_1\mu_1)_-,\quad \overline{L}_{\mu}=\frac{\lambda(\chi_2\mu_2-\chi_1\mu_1)_+}{\lambda-a\mu^2}, \quad$$
$$
K_{\mu}=\frac{|\chi_1\mu_1-\chi_2\mu_2|}{2}\Big(\frac{1}{\sqrt{\lambda-a\mu^2}}+\frac{\mu\sqrt{a}}{\lambda-a\mu^2} \Big), \quad M +\tilde M=|\chi_1\mu_1-\chi_2\mu_2|.
$$
\noindent{\bf Sub-case 1.} $|\chi_1\mu_1-\chi_2\mu_2|=0$.\\ In this case we have that $K_{\mu}=\overline{L}_{\mu}=0$ for every $0<\mu<\min\Big\{1, \sqrt{\frac{\lambda}{a}}\Big\}$. Hence, $\mu^*=\min\Big\{1, \sqrt{\frac{\lambda}{a}}\Big\}$ and
$$
c^{*}(\chi_1,\mu_1,\lambda,\chi_2,\mu_2,\lambda)=\begin{cases}2\sqrt{a}\quad \text{if}\quad a\leq \lambda\cr
\frac{a+\lambda}{\sqrt{\lambda}} \quad\ \text{if}\quad a\geq \lambda.
\end{cases}
$$
\noindent{\bf Sub-case 2.} $\chi_1\mu_1-\chi_2\mu_2>0$.\\ In this case we have that $\overline{L}_{\mu}=0$ for every $0<\mu<\min\Big\{1, \sqrt{\frac{\lambda}{a}}\Big\}$. 
Moreover, we have
$$
M +\tilde M+\chi_1\mu_1<b+\chi_2\mu_1 \Leftrightarrow |\chi_1\mu_1-\chi_2\mu_2|<\frac{b}{2}
$$
and
$$
\mu\sqrt{a}K_{\mu}+\overline{L}_{\mu}\leq b+\chi_2\mu_2-\chi_1\mu_1   \Leftrightarrow \mu\sqrt{a}\Big(\frac{1}{\sqrt{\lambda-a\mu^2}} +\frac{\mu\sqrt{a}}{\lambda-a\mu^2}\Big) \leq \frac{2(b-|\chi_1\mu_1-\chi_2\mu_1|)}{|\chi_1\mu_1-\chi_2\mu_2|}.
$$
Since the function $\Big(0\ ,\ \sqrt{\frac{\lambda}{a}}\Big)\ni\mu\mapsto \frac{\mu\sqrt{a}}{\sqrt{\lambda-a\mu^2}} +\frac{a\mu^2}{\lambda-a\mu^2}$ is strictly increasing and continuous, we have that
$$
\mu^*=\sup\Big\{\mu\in(0, \min\{1, \sqrt{\frac{\lambda}{a}}\})\ :\  \frac{\mu\sqrt{a}}{\sqrt{\lambda-a\mu^2}} +\frac{a\mu^2}{\lambda-a\mu^2}\leq \frac{2(b-|\chi_1\mu_1-\chi_2\mu_1|)}{|\chi_1\mu_1-\chi_2\mu_2|}   \Big\}.
$$
Hence
$$
\lim_{\chi_1\mu_1-\chi_2\mu_2\to 0^{+}}\mu^{*}=\min\{1, \sqrt{\frac{\lambda}{a}}\}.
$$
Which implies that
$$
\lim_{\chi_1\mu_1-\chi_2\mu_2\to 0^{+}}c^{*}(\chi_1,\mu_1,\lambda,\chi_2,\mu_2,\lambda)=\begin{cases}2\sqrt{a}\quad \text{if}\quad a\leq \lambda\cr
\frac{a+\lambda}{\sqrt{\lambda}} \quad\ \text{if}\quad a\geq \lambda.
\end{cases}
$$

\noindent{\bf Sub-case 3.} $\chi_2\mu_2-\chi_1\mu_1>0$.

\smallskip
 In this case we have that $\tilde M=0$.
Moreover, we have
$$
M+ \tilde M+\chi_1\mu_1<b+\chi_2\mu_1 \Leftrightarrow b>0
$$
and
$$
\mu\sqrt{a}K_{\mu}+\overline{L}_{\mu}\leq b+\chi_2\mu_2-\chi_1\mu_1   \Leftrightarrow\Big(\frac{\mu\sqrt{a}}{\sqrt{\lambda-a\mu^2}} +\frac{a\mu^2}{\lambda-a\mu^2}\Big) \leq \frac{2(b+|\chi_1\mu_1-\chi_2\mu_1|)}{|\chi_1\mu_1-\chi_2\mu_2|}
$$
Since the function $\Big(0 , \sqrt{\frac{\lambda}{a}}\Big)\ni\mu\mapsto \frac{\mu\sqrt{a}}{\sqrt{\lambda-a\mu^2}} +\frac{a\mu^2+\lambda}{\lambda-a\mu^2}$ is strictly increasing and continuous, we have that
$$
\mu^*=\sup\Big\{\mu\in(0, \min\{1, \sqrt{\frac{\lambda}{a}}\})\ \Big |\  \frac{\mu\sqrt{a}}{\sqrt{\lambda-a\mu^2}} +\frac{a\mu^2+\lambda}{\lambda-a\mu^2}\leq \frac{2(b+|\chi_1\mu_1-\chi_2\mu_1|)}{|\chi_1\mu_1-\chi_2\mu_2|}   \Big\}.
$$
Hence
$$
\lim_{\chi_2\mu_2-\chi_1\mu_1\to 0^{+}}\mu^{*}=\min\{1, \sqrt{\frac{\lambda}{a}}\}.
$$
Which implies that
$$
\lim_{\chi_1\mu_1-\chi_2\mu_2\to 0^{+}}c^{*}(\chi_1,\mu_1,\lambda,\chi_2,\mu_2,\lambda)=\begin{cases}2\sqrt{a}\quad \text{if}\quad a\leq \lambda\cr
\frac{a+\lambda}{\sqrt{\lambda}} \quad\ \text{if}\quad a\geq \lambda.
\end{cases}
$$
\end{proof}

\subsection{Proof of Theorem C}

In this subsection, we prove Theorem C.
To do so, we first recall the following two  lemmas from \cite{SaSh3}.

\begin{lem}
\label{nonexistence-lm1}
\begin{itemize}
\item[(1)]
Let $0\le c<2\sqrt a$ be  fixed  and $\lambda_0>0$ be such that $c^2-4 a+4\lambda_0<0$. Let $\lambda_D(L)$ be the principal eigenvalue of
\begin{equation}
\label{ev-eq0}
\begin{cases}
\phi_{xx}+c \phi_x +a\phi =\lambda \phi,\quad 0<x<L\cr
\phi(0)=\phi(L)=0.
\end{cases}
\end{equation}
Then there is $L>0$ such that  $\lambda_D(L)=\lambda_0$.

\item[(2)] Let $c$ and $L$ be as in (1). Let $\lambda_D(L;b_1,b_2)$ be the principal eigenvalue of
\begin{equation}
\label{ev-eq1}
\begin{cases}
\phi_{xx}+(c + b_1(x))\phi_x +(a+ b_2(x))\phi =\lambda \phi,\quad 0<x<L\cr
\phi(0)=\phi(L)=0,
\end{cases}
\end{equation}
where $b_1(x)$ and $b_2(x)$ are continuous functions.
 If there is a $C^2$ function $\phi(x)$ with $\phi(x)>0$ for
$0<x<L$  such that
\begin{equation}
\label{ev-eq2}
\begin{cases}
\phi_{xx}+(c + b_1(x))\phi_x +(a+ b_2(x))\phi  \le 0,\quad 0<x<L\cr
\phi(0)\ge 0 ,\quad \phi(L)\ge 0
\end{cases}
\end{equation}
 Then $\lambda_D(L,b_1,b_2)\le 0$.
\end{itemize}
\end{lem}



\begin{lem}
\label{nonexistence-lm2}
\begin{itemize}
\item[(1)] Let $c<0$ be fixed and let $\lambda_0>0$ be such that $0<\lambda_0<a$.
Let $\lambda_{N,D}(L)$ be the principal eigenvalue of
\begin{equation}
\label{ev-eq3}
\begin{cases}
\phi_{xx}+c \phi_x +a\phi =\lambda \phi,\quad 0<x<L\cr
\phi_x(0)=\phi(L)=0.
\end{cases}
\end{equation}
Then there is $L>0$ such that  $\lambda_{N,D}(L)=\lambda_0$.

\item[(2)] Let $c$ and $L$ be as in (1). Let $\lambda_{N,D}(L;b_1,b_2)$ be the principal eigenvalue of
\begin{equation}
\label{ev-eq4}
\begin{cases}
\phi_{xx}+(c + b_1(x))\phi_x +(a+ b_2(x))\phi =\lambda \phi,\quad 0<x<L\cr
\phi_x(0)=\phi(L)=0,
\end{cases}
\end{equation}
where $b_1(x)$ and $b_2(x)$ are continuous functions.
 If there is a $C^2$ function $\phi(x)$ with $\phi(x)>0$ for
$0<x<L$  such that
\begin{equation}
\label{ev-eq5}
\begin{cases}
\phi_{xx}+(c + b_1(x))\phi_x +(a+ b_2(x))\phi  \le 0,\quad 0<x<L\cr
\phi_x(0)\le  0 ,\quad \phi(L)\ge 0
\end{cases}
\end{equation}
 Then $\lambda_{N,D}(L,b_1,b_2)\le 0$.
\end{itemize}
\end{lem}



\begin{proof}[Proof of Theorem C]
We first consider the case that $0\le c<2\sqrt a$. Then there is $\lambda_0>0$ such that
$$
c^2-4 a+4\lambda_0<0.
$$
By Lemma \ref{nonexistence-lm1}(1), there is $L>0$ such that  $\lambda_D(L)=\lambda_0>0$.

Fix $0\le c<2\sqrt a$ and the above $L$.  Assume that   \eqref{Main-eq01} has a traveling wave solution
$(u,v_1,v_2)=(U(x-ct),V_1(x-ct),V_2(x-ct))$ with $(U(-\infty), V_1(-\infty),V_2(-\infty))=(\frac{a}{b},\frac{a\mu_1}{b\lambda_1},\frac{a\mu_2}{b\lambda_2})$ and
$(U(\infty),V_1(\infty),V_2(\infty))=(0,0,0)$.  Then, it follows that  \eqref{Main-eq1} has a stationary solution
$(u,v_1,v_2)=(U(x),V_1(x),V_2(x))$. Furthermore, we have  $(U(-\infty),V_1(-\infty),V_2(-\infty))=(\frac{a}{b},\frac{a\mu_1}{b\lambda_1},\frac{a\mu_2}{b\lambda_2})$ and  $(U(\infty),V_1(\infty),V_2(\infty))=(0,0,0)$.
Moreover,  for any $\epsilon>0$, this is $x_\epsilon>0$ such that
$$
0<U(x)<\epsilon, \quad 0<V_{1}(x)<\epsilon,\quad |V_{ix}(x)|<\epsilon \quad \forall \,\,i=1,2,\,\, x\ge x_\epsilon.
$$

Consider the eigenvalue problem,
\begin{equation}
\label{ev-eq6}
\begin{cases}
\phi_{xx}+(c+(\chi_2V_2-\chi_1 V_1)_x)\phi_x+(a+(\chi_2\lambda_2V_2-\chi_1\lambda_1V_1)(x))\phi\cr
 -((b+\chi_2\mu_2-\chi_1\mu_1)U(x))\phi=\lambda\phi,\quad x_\epsilon<x<x_\epsilon+L\cr
\phi(x_\epsilon)=\phi(x_\epsilon+L)=0.
\end{cases}
\end{equation}
Let $\lambda_D^\epsilon(L)$ be the principal eigenvalue of \eqref{ev-eq6}.
By Lemma \ref{nonexistence-lm1}(1) and perturbation theory for principal eigenvalues of elliptic operators,  $\lambda_D^\epsilon(L)>0$ for $0<\epsilon\ll 1$.

Note that
$$
U_{xx}+(c +(\chi_2V_2-\chi_1V_1)_{x})U_{x} + (a+(\chi_2\lambda_2V_2-\chi_1\lambda_1V)(x)-(b-\chi)U(x))U=0\quad \forall\,\, x_\epsilon\le x\le x_\epsilon+L
$$
and $U(x_\epsilon)>0$, $U(x_\epsilon+L)>0$. Then, by Lemma \ref{nonexistence-lm1}(2),  $\lambda_D^\epsilon(L)\le 0$. We get a contradiction.
Therefore,  \eqref{Main-eq01} has no traveling wave solution
$(u,v_1,v_2)=(U(x-ct),V_1(x-ct),V_2(x-ct))$ with $(U(-\infty),V_1(-\infty),V_2(-\infty))=(\frac{a}{b},\frac{a\mu_1}{b\lambda_1},\frac{a\mu_2}{b\lambda_2})$ and  $(U(\infty),V_1(\infty),V_2(\infty))=(0,0,0)$ and $0\le c<2\sqrt a$.

\smallskip

Next, we consider the case that $c<0$. Let $\lambda_0$ and $L$ be as in Lemma \ref{nonexistence-lm2}(1).
Then $\lambda_{N,D}(L)=\lambda_0>0$.

Fix $c<0$ and the above $L$.
 Assume that   \eqref{Main-eq01} has a traveling wave solution
$(u,v_1,v_2)=(U(x-ct),V_1(x-ct),V_2(x-ct))$ with $(U(-\infty)$, $V(-\infty))=(a/b,a/b)$ and
$(U(\infty),V(\infty))=(0,0)$.
Then, it follows that \eqref{Main-eq1} has a stationary solution
$(u,v_1,v_2)=(U(x),V_1(x),V_2(x))$ with $(U(-\infty),V_1(-\infty),V_2(-\infty))=(\frac{a}{b},\frac{a\mu_1}{b\lambda_1},\frac{a\mu_2}{b\lambda_2})$ and  $(U(\infty),V_1(\infty),V_2(\infty))=(0,0,0)$.
Similarly, for any $\epsilon>0$, this is $x_\epsilon>0$ such that
$$
0<U(x)<\epsilon, \quad 0<V_i(x)<\epsilon,\quad |V_{ix}(x)|<\epsilon \quad \forall \,\, i=1,2,\,\, x\ge x_\epsilon.
$$
Moreover,
since $U(\infty)=0$, there is $\tilde x_\epsilon>x_\epsilon$ such that
$$
U_x(\tilde x_\epsilon)<0.
$$

Consider the eigenvalue problem,
\begin{equation}
\label{ev-eq7}
\begin{cases}
\phi_{xx}+(c+(\chi_2V_2-\chi_1V_1)_x(x))\phi_x+(a+(\chi_2\lambda_2V_2-\chi_1\lambda_1V_1)(x))\phi\cr
 -((b+\chi_2\mu_2-\chi_1\mu_1)U(x))\phi=\lambda\phi,\quad \tilde x_\epsilon<x<\tilde x_\epsilon+L\cr
\phi_x(\tilde x_\epsilon)=\phi(\tilde x_\epsilon+L)=0.
\end{cases}
\end{equation}
 Let $\lambda_{N,D}^\epsilon(L)$ be the principal eigenvalue of \eqref{ev-eq7}.
By Lemma \ref{nonexistence-lm2}(1) and perturbation theory for principal eigenvalues of elliptic operators,  $\lambda_{N,D}^\epsilon(L)>0$ for $0<\epsilon\ll 1$.

Note that
$$
U_{xx}+(c +(\chi_2V_2-\chi_1V_1)_{x})U_{x} + (a+(\chi_2V_2-\chi_1V_1)(x)-(b+\chi_2\mu_2-\chi_1\mu_1)U(x))U=0\quad \forall\,\, \tilde x_\epsilon\le x\le\tilde  x_\epsilon+L
$$
and $U_x(\tilde x_\epsilon)<0$, $U(\tilde x_\epsilon+L)>0$. Then, by Lemma \ref{nonexistence-lm2}(2),  $\lambda_{N,D}^\epsilon(L)\le 0$. We get a contradiction.
Therefore,  \eqref{Main-eq01} has no traveling wave solution
$(u,v_1,v_2)=(U(x-ct),V_1(x-ct),V_2(x-ct))$ with $(U(-\infty),V_1(-\infty),V_2(-\infty))=(\frac{a}{b},\frac{a\mu_1}{b\lambda_1},\frac{a\mu_2}{b\lambda_2})$ and  $(U(\infty),V_1(\infty),V_2(\infty))=(0,0,0)$ and $c<0$.

Theorem C is thus proved.
\end{proof}

\end{document}